\documentclass[12pt]{amsart}
 \usepackage{amsmath,amssymb,amsfonts,amsthm,amscd,textcomp,times,booktabs,mathrsfs}
   \usepackage{graphicx}
    \usepackage{braket}
    \usepackage{tabularx}
  \usepackage{color,float}
  \usepackage{bm}
\usepackage[all]{xy}
 \usepackage{wrapfig}

\newtheorem{theorem}{Theorem}

\newtheorem{claim}[theorem]{Claim}
\newtheorem{proposition}[theorem]{Proposition}
\newtheorem{lemma}[theorem]{Lemma}
\newtheorem{definition}[theorem]{Definition}
\newtheorem{corollary}[theorem]{Corollary}
\newtheorem{remark}[theorem]{Remark}
\newtheorem{examplex}{Example}
\newtheorem{problem}{Problem}

\numberwithin{equation}{section}
\numberwithin{theorem}{section}

\newcommand{\N}{\ensuremath{\mathbb{N}}}
\newcommand{\Z}{\ensuremath{\mathbb{Z}}}

\newcommand{\R}{\ensuremath{\mathbb{R}}}
\newcommand{\C}{\ensuremath{\mathbb{C}}}
\newcommand{\Ch}{\ensuremath{\mathrm{Ch}}}

\newcommand{\Hu}{\ensuremath{\mathrm{Hu}}}

\newcommand{\PL}{\ensuremath{\mathrm{PL}}}
\newcommand{\SL}{\ensuremath{\mathrm{SL}}}

\newcommand{\Wt}{\ensuremath{\mathrm{Wt}}}

\newcommand{\bs}{\ensuremath{\boldsymbol}}

\newcommand{\reduce}[1]{\scalebox{1}{\ensuremath{#1}}}
\newcommand{\vol}{\ensuremath{\mathrm{vol}}}

\newcommand{\bracket}[1]{\ensuremath{\langle #1 \rangle}}

\newcommand{\Norm}[1]{\ensuremath{\left\| #1 \right\|}}
\newcommand{\restrict}[2]{\ensuremath{\left. #1 \right|_{#2}}}

\newenvironment{example}
  {\pushQED{\qed}\examplex}
  {\popQED\endexamplex}

\DeclareMathOperator{\Hom}{Hom}

\DeclareMathOperator*{\smalloplus}{\reduce{\bigoplus}}

\DeclareMathOperator*{\smallcup}{\reduce{\bigcup}}

\def\C{\mathbb C}
\def\D{\Delta}

\def\G{\mathbb{G}}

\def\P{\mathbb{P}}
\def\R{{\mathbb R}}

\def\V{{\boldsymbol{V}}}
\def\W{{\boldsymbol{W}}}
\def\X{\mathcal{X}}

\def\codim{{\mathrm{codim\,}}}
\def\conv{{\rm{Conv}}}
\def\p{\partial }
\def\sec{\displaystyle \Sigma{\rm{sec}}}
\def\Hu{\mathrm{Hu}}
\def\vp{\varphi}

\begin{document}

\title{Numerical semistability of projective toric varieties}

\author{Naoto Yotsutani}

\address{Kagawa University, Faculty of education, Mathematics, Saiwaicho $1$-$1$, Takamatsu, Kagawa,
$760$-$8522$, Japan}
\email{yotsutani.naoto@kagawa-u.ac.jp}

\makeatletter
\@namedef{subjclassname@2020}{%
  \textup{2020} Mathematics Subject Classification}
\makeatother

\subjclass[2020]{Primary 51M20; Secondary 14M25, 53C55}

\keywords{Secondary polytope, GKZ-theory, $A$-discriminants, Hurwitz form}

\maketitle

\noindent{\bfseries Abstract.}
Let $X \to \P^N$ be a smooth linearly normal projective variety. It was proved by Paul that the $K$-energy of $(X,\restrict{\omega_{FS}}{X})$ restricted to the Bergman metrics is bounded from below if and only if the pair 
of (rescaled) Chow/Hurwitz forms of $X$ is {\emph{numerically semistable}}. 
In this paper, we provide a necessary and sufficient condition for a given smooth toric variety $X_P$ to be numerically semistable with respect to $\mathcal O_{X_P}(i)$ for a positive integer $i$.
Applying this result to a smooth polarized toric variety $(X_P, L_P)$, we prove that $(X_P, L_P)$ is asymptotically numerically semistable if and only if it is K-semistable for toric degenerations.


\section{Introduction}
This paper aims to provide a necessary and sufficient condition of 
numerical semistability defined in \cite{Sean12} for a smooth projective toric variety by adapting Gelfand-Kapranov-Zelevinsky theory
($A$-discriminants/resultants, Hurwitz forms). Surprisingly, we will see that our problem can be deduced into the beautiful binomial theorem in probability theory (see, Claim \ref{claim:binomial}).  

Firstly, we recall the notion of numerical semistability and semistability of pairs defined in \cite{Sean12}.
Let $G$ be a complex reductive algebraic group. Let $\bs V$ and $\bs W$ be finite dimensional $G$-representations with nonzero vectors $v\in \bs V\setminus \set{0}$ and $w\in \bs W\setminus \set{0}$. 
Considering the two $G$-orbits
\begin{align}
\begin{split}\label{eq:orbits}
\mathcal O_G(v,w)&:=G\cdot[(v,w)]\subset \P(\bs V\oplus \bs W), \quad \text{and} \\
\mathcal O_G(v)&:=G\cdot [(v,0)]\subset \P(\bs V\oplus \set{0})
\end{split}
\end{align}
respectively, we denote by $\overline{\mathcal O_G(v,w)}$ (resp. $\overline{\mathcal O_G(v)}$) the Zariski closure of $\mathcal O_G(v,w)$ in $\P(\bs V\oplus \bs W)$ 
(resp. of $\mathcal O_G(v)$ in $\P(\bs V\oplus \set{0})$). The pair $(v,w)$ is said to be {\emph{semistable}} if $\overline{\mathcal O_G(v,w)} \cap \overline{\mathcal O_G(v)} =\emptyset$.
In view of Hilbert-Mumford criterion in classical GIT, it is essential to study semistability when a reductive algebraic group is isomorphic to an algebraic torus 
(see, Proposition \ref{prop:HM_criterion}). For a fixed maximal algebraic torus $H\leq G$, let $\Wt_H(v)$ (resp. $\Wt_H(w)$) be the {\emph{weight polyotpe}} of $v$ (resp. $w$). 
Then, the pair $(v,w)$ is called {\emph{numerically semistable}} if $\Wt_H(v) \subseteq \Wt_H(w)$ for any maximal algebraic torus $H\leq G$ (Definition \ref{def:Pss}).
Obviously, if the pair $(v,w)$ is semistable then it is numerically semistable, although the converse is not true for general pair $(v,w)$ as in \cite{PSZ23}.  
However, we remark that the implication 
\begin{center}
numerical semistablity \quad $\Rightarrow$ \quad semistability (of pairs)
\end{center}
could possibly hold for pairs of the form $(R_X^{\deg \Hu_X}, \Hu_X^{\deg R_X} )$ below.

Secondly, we recall the historical background of semistability of pairs.
Let $X^n \to \P^N$ be an $n$-dimensional irreducible smooth complex projective variety of degree $d\geqslant 2$. An irreducible subvariety $X\subset \P^N$ is called {\emph{non-degenerate}} if $X$ is not contained in any
hyperplane. A non-degenerate projective variety $X\subset \P^N$ is called \emph{linearly normal} if it 
cannot be represented as an isomorphic projection of a non-degenerate projective variety from a higher dimensional projective space. In order to study projective dual varieties and discriminants, it suffices to consider linearly normal projective varieties (see, Theorems $1.21$ and $1.23$ in \cite{Tev05}). Here and hereafter, we assume that a complex projective variety $X$ is linearly normal.
For the Fubini-Study K\"ahler form $\omega_{FS}$ of $\P^N$, we set $\omega$ by $\restrict{\omega_{FS}}{X}$. 

Let $G$ be the special linear group $\mathrm{SL}(N+1,\C)$ which naturally acts on $\P^N$. For each $\sigma \in G$, the associated {\emph{Bergman potential}} $\varphi_\sigma \in C^\infty(X)$ is given by
\[
\sigma^*\omega=\omega+\frac{\sqrt{-1}}{2\pi}\partial \overline \partial \varphi_\sigma >0.
\]
We denote by $\nu_\omega$ the K-energy of $(X,\omega)$.
For a complex projective variety $X^n \subset \P^N$ of degree $d\geqslant 2$, we denote the Chow form of $X$ by $R_X$.
Let $\P_N^\vee$ be the dual projective space of $\P^N$. The dual variety of $X$, which is the variety of tangent hyperplanes to $X$, is denoted by $X^\vee \subset \P_N^\vee$.
Let $\delta(X)$ denote the {\emph{dual defect}} of $X\hookrightarrow \P^N$ which is the nonnegative integer defined by
\[
\delta(X):=N-1-\dim X^\vee.
\]
If the dual variety $X^\vee$ is a hypersurface in $\P_N^\vee$, then the unique defining polynomial (up to a multiplicative constant) denoted by $\D_X$ is called the {\emph{$X$-discriminant}}: 
\[
X^\vee=\set{f\in \P_N^\vee | \D_X(f)=0}.
\]

Next we consider the Segre embedding $X^n\times \P^{n-1}\to \P^{(n+1)n-1}=:\P^\bullet$ for an $n$-dimensional projective variety $X^n\subset \P^N$.
The Segre variety $Y$ of $X^n\times \P^{n-1}$ in $\P^\bullet$ has the dual defect $\delta(Y)=0$, that is, the dual variety $Y^\vee$ is \emph{always} a hypersurface in $\P_\bullet^\vee$ according to \cite{WZ94}
(see also, \cite[p. $271$]{Sean12}). Then the defining polynomial of $Y^\vee$ is called the {\emph{$X$-hyperdiscriminant form}} which will be denoted by $\D_{X\times \P^{n-1}}$ throughout the paper.
As we shall see in Section \ref{sec:Hurwitz}, it was proved in \cite{St17} that the hyperdiscriminant form $\D_{X\times \P^{n-1}}$ coincides with the Hurwitz form $\Hu_X$ in the Stiefel coordinates.
See, Section \ref{sec:Hurwitz} for the definition of the Hurwitz forms.   

In \cite{Sean12}, Paul detected that the asymptotic behavior of the $K$-energy map along the Bergman potentials is expressed by the log norms of the Chow/Hurwitz forms.
\begin{theorem}[Theorem A in \cite{Sean12}]\label{thm:Sean1}
Let $X\subset \P^N$ be a linearly normal smooth projective variety of degree $d\geqslant 2$, where $G=\SL(N+1,\C)$ acts to $\P^N$ linearly.  
Let $R_X$ (resp. $\Hu_X$) be the Chow form (resp. the Hurwitz form) of $X$.
For any $\sigma \in G$, there are norms $\Norm{\cdot}$ such that the $K$-energy restricted to the Bergman metrics is given as follows:
\begin{equation}\label{eq:K-enAsymp}
\nu_\omega(\varphi_\sigma)=\deg(R_X)\log\frac{\Norm{\sigma \cdot \Hu_X}^2}{\Norm{ \Hu_X}^2}
-\deg(\Hu_X)\log\frac{\Norm{\sigma \cdot R_X}^2}{\Norm{ R_X}^2}.
\end{equation}
\end{theorem}
Let $H$ be a maximal algebraic torus of $G$. Let $\Wt_H(R_X)$ (resp. $\Wt_H(\Hu_X)$) be the weight polytope of the Chow form $R_X$ (resp. the Hurwitz form $\Hu_X$) with respect to $H$. 
Through $\eqref{eq:K-enAsymp}$, we see that the asymptotic behavior of the $K$-energy map along any one parameter subgroup of $H$ is dominated by $\Wt_H(R_X)$ and  $\Wt_H(\Hu_X)$. 
More precisely, Paul proved the following.
\begin{theorem}[Theorem C in \cite{Sean12}]\label{thm:Sean2}
Let $\omega:=\restrict{\omega_{FS}}{X}$ and let $\nu_\omega$ be the $K$-energy of $(X,\omega)$.
Then $\nu_\omega$ is bounded from below along any degenerations in $G$ if and only if for any maximal torus $H$ in $G$, the corresponding weight polytope of the Hurwitz form $\Hu_X$ dominates
the weight polytope of the Chow form $R_X$ in the sense that
\begin{equation}\label{eq:Pss}
\deg(\Hu_X)\Wt_H(R_X)\subseteq \deg(R_X)\Wt_H(\Hu_X).
\end{equation}
\end{theorem}
We call a smooth projective variety $X\subset \P^N$ is {\emph{numerically semistable}} 
if the associated two weight polytopes satisfy the inclusion $\eqref{eq:Pss}$. Equivalently, we say that the pair $(R_{X}^{\deg \Hu_{X}}, \Hu_{X}^{\deg R_{X}})$ is numerically semistable
(see, Definitions \ref{def:Pss} and \ref{def:Pss2}).
The main focus of this paper is to investigate a necessary and sufficient condition for the inclusion in $\eqref{eq:Pss}$ for any smooth projective toric variety, building upon the works of
GKZ-theory \cite{GKZ94, WZ94, BFS90, KSZ92, St17}.

In order to state our main theorem more specifically, let $A=\set{\bs a_0, \ldots ,\bs a_N}$ be a finite set of points in $M_\Z\cong \Z^n$. 
Here and hereafter, we denote by $P$ the convex hull of $A$ in $\R^n(\cong M_{\R})$. 
A finite set $A$ is said to {\emph{satisfy $(*)$}} if the following conditions hold:
\begin{enumerate}
\item[(i)] $A=P\cap \Z^n=\set{\bs a_0, \ldots ,\bs a_N}$, and
\item[(ii)] $A$ affinely generates the lattice $\Z^n$ over $\Z$.
\end{enumerate}
By shrinking the lattice $M_\Z$ 
if necessary, we can assume that a finite set $A$ satisfies $(*)$. We regard $A$ as a set of Laurent monomials in $n$ variables, i.e., of monomials of the form
\[
\bs x^{\bs a}=x_1^{a_1}\cdots x_n^{a_n},
\]
where $\bs a=(a_1, \ldots ,a_n)\in A$ is the exponent vector and $x_1, \ldots ,x_n$ are $n$-variables.
The closure of the $A$-monomial embedding of a complex torus $(\C^\times)^n$ to the projective space $\P^N$ defines the $n$-dimensional projective toric variety $X_P\subset \P^N$.
Remark that $P=\conv(A)$ in $\R^n\cong M_\R$ is the corresponding moment (Delzant) polytope of $X_P$ with respect to the natural $(S^1)^n$-action which determines the toric structure of $X_P$. 
Now we clarify the problem that we shall consider in the present paper.
\begin{problem}\rm
Let $X_P\subset \P^N$ be a linearly normal smooth projective toric variety. 
Provide a necessary and sufficient condition for $X_P$ to be 
numerically semistable in terms of the associated moment polytope $P$.
\end{problem}
Conjecturally this problem is closely related to K-semistability in the sense of Donaldson \cite{Don02} in ``finite dimensional" setup. 
In our approach, we only use GKZ technique for investigating 
numerical semistability defined by $\eqref{eq:Pss}$ for a projective toric variety $X_P$.
For $i\in \N$, let $\PL(P;i)$ be the set of concave piecewise linear functions on $P$ defined in $\eqref{eq:PLfuns}$. 
Let $\mathcal W(X_P)$ the Weyl group of $X_P$ and $\PL(P;i)^{\mathcal W(X_P)}$ be the $\mathcal W(X_P)$-invariant subset of $\PL(P;i)$.
See, $\eqref{eq:PLinv}$ for further details on the terminologies. One of the main results in this paper is the following.
\begin{theorem}\label{thm:main}
Let $A=\set{\bs a_0, \dots , \bs a_N}\subset M_\Z$ be a finite set of lattice points which satisfies $(*)$.
Let $P$ denote the convex hull of $A$ in $M_\R$. For a fixed positive integer $i$, we denote the Chow form (resp. the Hurwitz form) of $(X_P, L_P^{\otimes i})$ by $R_{X_P}(i)$ \big(resp. $\Hu_{X_P}(i)$ \big).
Let $N_i$ be the number of lattice points in $iP$. Then the pair
$\left( R_{X_P}(i)^{\deg \Hu_{X_P}(i)}, \Hu_{X_P}(i)^{\deg R_{X_P}(i)}\right)$ is numerically semistable 
with respect to the natural action of $\mathrm{SL}(N_i,\C)$ if and only if
\begin{equation}\label{ineq:FPinv}
F_P(i;g_\varphi):=\vol(\p P)\int_Pg_\vp(\bs x)dv-\vol( P)\int_{\p P}g_\vp(\bs x)d\sigma \geqslant 0
\end{equation}
for each $g_\vp\in \PL(P;i)^{\mathcal W(X_P)}$.
\end{theorem}
Remark that we {\emph{only}} consider the concave piecewise linear functions which determines the regular triangulations of $(P,A)$ in Theorem \ref{thm:main}.
In other words, it suffices to see the nonegativity in $\eqref{ineq:FPinv}$ for degenerations corresponding to the vertices of the secondary polytope, in particular only finitely many degenerations (cf. \cite[Chapter 7]{GKZ94}). 
Applying Theorem \ref{thm:main} to a smooth polarized toric variety, we prove that asymptotic numerical semistability is equivalent to K-semistability.
\begin{corollary}[See, Theorem \ref{thm:PandK}]\label{cor:PandK}
Let $(X_P, L_P)$ be a smooth polarized toric variety with the associated moment polytope $P$.
Then $(X_P,L_P)$ is asymptotically numerically semistable if and only if it is K-semistable for toric degenerations.
\end{corollary}
This paper is organized as follows. Section \ref{sec:Prelim} collects the basic definitions and terminologies used in the paper.
Section \ref{sec:GKZ} gives the fundamental properties of the Chow/Hurwitz polytopes of projective toric varieties.
In Section \ref{sec:combinat}, we calculate the degree of the Hurwitz form by adapting the $A$-discriminant formula $\eqref{eq:degDiscrim}$ and Vandermonde's identity.
This yields the key idea for computing the affine hull of a Hurwitz polytope and the Futaki-Paul invariant for a toric variety in Section \ref{sec:KeyProps}.
We prove Theorem \ref{thm:main} and Corollary \ref{cor:PandK} in Sections \ref{sec:Proof} and \ref{sec:PandK}, respectively.

Through the paper, we work over a complex number field $\C$, and a {\emph{manifold}} is a smooth irreducible complex projective variety.

\vskip 4pt

\noindent{\bf{Acknowledgement.}}
The author would like to thank Professors Chi Li and Yuji Sano for fruitful discussions on this topic.
He also thanks Professor Sean Paul to point him out the relationship between numerical semistability and semistability of pairs in plain words. 
This work was partially supported by JSPS KAKENHI Grant Numbers JP$18$K$13406$ and JP$22$K$03316$.

\section{Preliminaries}\label{sec:Prelim}
In this section, we recall background and some basic facts about GIT-stability. In order to deal with semistability  of pairs for a projective variety $X\subset \P^N$, we firstly define 
the notion of {\emph{semistable pair}} for any two nonzero vectors $v$ and $w$ (Definition $\ref{def:Pss}$). 
Later on, we will adapt this concept to the pair of the Chow form $R_X$ and the Hurwitz form $\Hu_X$ 
following Paul's strategy \cite{Sean12, Sean13, Sean21}.

\subsection{Hilbert-Mumford's semistability}\label{sec:GIT}
Let $G$ be a reductive algebraic group and $\V$ a finite dimensional complex vector space. Suppose $G$ acts linearly on $\V$.
A nonzero vector $v\in \V$ is called {\emph{$G$-semistable}} if the Zariski closure of the $G$-orbit $\mathcal O_G(v)$ in $\V$ does not contain the origin: $0\notin \overline{\mathcal O_G(v)}$.
For any $v\in \V\setminus \set{ 0}$, let $[v]\in\P(\V)$ be its projectivization. For simplicity, we abbreviate $[v]$ by $v$. Then $v\in \P(\V)$ is said to be $G$-semistable if any representation of $v$ is
$G$-semistable. There is a good criterion to verify a given nonzero vector $v$ being $G$-semistable.
\begin{proposition}[Hilbert-Mumford criterion]\label{prop:HM_criterion}
A point $v\in \P(\V)$ is $G$-semistable if and only if it is $H$-semistable for all maximal algebraic torus $H\leq G$.
\end{proposition}

Now we assume that a reductive algebraic group $G$ is isomorphic to an algebraic torus.
Let $\mathcal X(G)$ denote the character group of $G$. Then $\X(G)$ consists of algebraic homomorphisms $\chi:G\to \C^\times$. 
If we fix an isomorphism $G\cong (\C^\times)^{N+1}$, we may express each $\chi$ as a Laurent monomial
\[
\chi(t_0, \ldots ,t_N)=t_0^{a_0}\cdots t_N^{a_N} \qquad \text{for} \quad t_i\in\C^{\times} \quad \text{and} \quad a_i\in \Z.
\]
Hence there is the identification between $\X(G)$ and $\Z^{N+1}$:
\[
\chi=(a_0, \ldots ,a_N)\in \Z^{N+1}.
\]
Then, it is known that $\V$ decomposes under the action of $G$ into weight spaces
\[
\V=\smalloplus_{\chi\in \X(G)}\V_\chi, \qquad \V_\chi:=\Set{v\in \V|t\cdot v=\chi(t)\cdot v, \quad t\in G}.
\]

Let $v\in \V\setminus\set{ 0}$ be a nonzero vector in $\V$ with
\[
v=\sum_{\chi \in \X(G)}v_\chi,\qquad v_\chi\in \V_\chi.
\]
For a finite set of integer vectors $A$, we denote the convex hull of $A$ by $\conv(A)$. The {\emph{weight polytope}} of $v$ with respect to $G$-action is the convex lattice polytope in
$\X(G)\otimes \R\cong \R^{N+1}$ defined by
\[
\Wt_G(v):=\conv\set{\chi \in \chi(G)|v_\chi \neq \bs 0}\subset \R^{n+1}.
\]
In the case where $G$ is an algebraic torus, the Hilbert-Mumford criterion (Proposition $\ref{prop:HM_criterion}$) can be restated as follows.
For a proof, we refer the reader to \cite[Theorem $9.2$]{Dol03}.
\begin{proposition}[The numerical criterion]\label{prop:numerical}
Suppose $G$ is isomorphic to an algebraic torus which acts on a complex vector space $\V$ linearly.
For a nonzero vector $v$ in $\V$, $v$ is $G$-semistable if and only if $\Wt_G(v)$ contains the origin.
\end{proposition}
Now let $G_v$ be the stabilizer for a fixed point $v\in \V$.
Let $G_0$ be a reductive subgroup of $G_v$. Let us denote the centralizer of $G_0$ in $G$ by $C(G_0)$.
Then we have the natural $C(G_0)$-action on $\V$ induced by the injective map $C(G_0)\hookrightarrow G$.
The following Kempf's Instability Theorem was investigated by Li-Li-Sturm-Wang.
\begin{proposition}[Corollary $2.3$ in \cite{LLSW19}]\label{prop:Kempf}
A point $v \in \P(\V)$ is $G$-semistable if and only if it is $C(G_0)$-semistable.
\end{proposition}

\subsection{Semistability of pairs}\label{sec:Pss}
Next we recall the definitions of numerical semistability and semistablity of pairs which were introduced by Paul in \cite{Sean12, Sean13}. 
In particular, these concepts can be regarded as natural generalizations of Hilbert-Mumford's semistability discussed in 
Section $\ref{sec:GIT}$ (see, Proposition $\ref{prop:GITss}$). For more details, we refer the reader to Section $2$ in \cite{Sean12} and \cite{Sean13}.

Let $G$ be 
a complex reductive algebraic group and let $(\V,\rho)$ be a rational $G$-representation with $v\in \V\setminus \set{ 0}$.
We recall that $\bs V$ is said to be {\em{rational}} if for any $v\in \V\setminus \set{ 0}$ and $\alpha \in \V^{\vee}$ (the dual vector space)
\[
\varphi_{\alpha,v}:G\longrightarrow \C, \qquad \varphi_{\alpha,v}(\sigma):=\alpha(\rho(\sigma)\cdot v)
\]
is a regular function on $G$, i.e.,
\[
\varphi_{\alpha,v}\in \C[G]:=\text{affine coordinate ring of $G$.}
\]
Let $H$ be a maximal algebraic torus of $G$. We denote the character lattice of $H$ by
\[
{M}'_\Z:=\Hom_\Z(H,\C^\times)\cong \Z^N.
\]
Then the dual lattice ${N}'_\Z:=\Hom_\Z(\C^\times,H)$ is identified with the set of one parameter subgroups $\lambda:\C^\times \to H$.
Note that the duality is given by
\[
\braket{\, \cdot , \cdot\,}:{N'_\Z}\times{M}'_\Z \longrightarrow \Z, \qquad \chi(\lambda(t))=t^{\braket{\lambda,\chi}}.
\]
As usual, we set
\[
{M'_\R}:={M}'_\Z\otimes_\Z\R
\qquad \text{and} \qquad {N}'_\R:={N'_\Z}\otimes_\Z\R.
\]
Denoting the image of $\lambda$ in $N'_\R$ by $l_\lambda$, we see that $l_\lambda$ is an integral linear functional on $M'_\R$.
The key to understanding semistability of pairs in terms of Proposition $\ref{prop:numerical}$ is the following concept.
\begin{definition}\rm
Let $\V$ be a rational representation of $G$, and let $H$ be a maximal algebraic torus of $G$.
Let us denote a one parameter subgroup in $H$ by $\lambda$. We define the {\emph{weight}} $w_\lambda(v)$ of $\lambda$ on $v\in \V\setminus \set{0}$ by
\[
w_\lambda(v):=\min_{\bs x\in \Wt_H(v)}l_\lambda(\bs x)=\min\set{\braket{\chi,\lambda}|\chi \in \mathrm{supp }(v)}.
\]
\end{definition}
Let $\V$ and $\W$ be rational $G$-representations with nonzero vectors $v \in \V\setminus \set{0}$ and $w \in \W\setminus \set{ 0}$. 
As we saw in $\eqref{eq:orbits}$, we consider the $G$-orbits $\mathcal O_G(v,w)$ in $\P(\bs V\oplus \bs W)$ and $\mathcal O_G(v)$ in $\P(\bs V\oplus \set{0})$ respectively,
and denote by $\overline{\mathcal O_G(v,w)}$ and $\overline{\mathcal O_G(v)}$ their closures in the corresponding projective spaces.
\begin{definition}\label{def:ss}\rm
The pair $(v,w)$ is said to be {\emph{semistable}} if $\overline{\mathcal O_G(v,w)}\cap \overline{\mathcal O_G(v)}=\emptyset$.
\end{definition}
For convenience, we call $(v,w)$ is {\emph{P-semistable}}\footnote{The name of P-stability is derived from the words from (i) {\em P}aul and (ii) Stability of {\em P}airs.}
if the condition in Definition \ref{def:ss} holds.
\begin{definition}\label{def:Pss}\rm
The pair $(v,w)$ is called {\emph{numerically semistable}} if  $\Wt_H(v)\subseteq \Wt_H(w)$ for any maximal algebraic torus $H\leq G$.
\end{definition}
It was also shown in \cite[Proposition $2.5$]{Sean12} that $(v,w)$ is numerically semistable if and only if $w_{\lambda}(w)\leqslant w_{\lambda}(v)$ for all degenerations $\lambda$ in $H$,
where $H$ is a maximal algebraic torus of $G$. We remark that the meaning of {\emph{semistable}} in Definitions \ref{def:ss} and \ref{def:Pss} originates from the following fact. 
Namely, if we consider the case where $\V$ is the trivial one dimensional $G$-representation $(\V\cong \C)$,
then the straightforward computations show that P-semistability (resp. numerical semistability) coincides with Hilbert-Mumford's semistability in classical GIT.
\begin{proposition}[Example $2$ in \cite{Sean12}]\label{prop:GITss}
The pair $(1,w)$ is P-semistable if and only if the Zariski closure of $\mathcal O_G(w)$ does not contain the origin:
 $0\notin \overline{\mathcal O_G(w)}$. Moreover, $(1,w)$ is numerically semistable if and only if $0\in \Wt_H(w)$ for any maximal algebraic torus $H\leq G$.
\end{proposition}
In Definition \ref{def:Pss2}, we adapt the concept of numerical semistability to liniearly normal projective varieties.
\begin{remark}\label{rem:PSZ}\rm
From Definitions \ref{def:ss} and \ref{def:Pss}, one can show that numerical semistability implies P-semistability for {\emph{any}} pair $(v,w)$ in arbitrary pair of $G$-representations
$(\bs V, \bs W)$. Although the converse is not true in general \cite{PSZ23, BHJ22}, it could be true for the pair $(R_X^{\deg \Hu_X},\Hu_X^{\deg R_X})$.
For reader's convenience, we clarify the relation between the inclusion of weight polytopes and disjoint orbit closures as follows.

Let $(v,w)$ be the pair of nonzero vectors in $G$-representations. Then, the following three conditions are equivalent:
\begin{enumerate}
\item $\Wt_H(v)\subseteq \Wt_H(w)$ for any maximal algebraic torus $H\leq G$.
\item $\overline{\mathcal O_H(v,w)}\cap \overline{\mathcal O_H(v)}=\emptyset$ for any maximal algebraic torus $H\leq G$.
\item $w_\lambda(w)\leqslant w_\lambda(v)$ for any degeneration $\lambda$ in $H$ (a maximal algebraic torus of $G$).
\end{enumerate}
\end{remark}

\subsection{The Chow forms and the discriminant forms}\label{sec:CD}
This subsection collects the definitions of the Chow forms and the discriminant forms of irreducible complex projective varieties.
See, \cite{GKZ94} for more details on these topics.

Let $X^n\to \P^N$ be an $n$-dimensional irreducible complex projective variety of degree $d\geqslant 2$.
We always assume that $X$ is {\emph{smooth}} unless otherwise stated.
Let $\mathbb G(k,N)$ be the Grassman variety that parameterizes $k$-dimensional projective linear subspaces of $\P^N$.
The {\emph{associated hypersurface}} $Z_X$ of $X^n\to \P^N$ is the subvariety in $\G(N-n-1,N)$ given by
\begin{equation}\label{def:AssociateHyp}
Z_X:=\Set{L\in \G(N-n-1,N)| L\cap X\neq \emptyset}.
\end{equation}
We note that $Z_X$ has the following fundamental properties (see, \cite[p.$99$]{GKZ94}):
\begin{enumerate}
\item $Z_X$ is an irreducible divisor in $\G(N-n-1,N)$,
\item $\deg Z_X=d$ in the Pl\"ucker coordinates, and
\item $Z_X$ is given by the vanishing of a section $R_X\in H^0(\G(N-n-1,N),\mathcal O(d))$.  
\end{enumerate}
The {\emph{Cayley-Chow form ($X$-resultant)}} of $X^n\to \P^N$, denoted by $R_X$, is the defining polynomial (unique up to constant multiples)
of the irreducible divisor $Z_X$ defined in $\eqref{def:AssociateHyp}$. Remark that $R_X$ is also irreducible polynomial with
$\deg R_X=d$ in the Pl\"ucker coordinates. Setting $\V:=H^0(\G(N-n-1,N),\mathcal O(d))$ and $R_X\in \P(\V)$ which is the projectivization of the Chow form, we call $R_X$ the {\emph{Chow point}} of $X$. In particular, the weight polytope 
of the Chow point (form) of $X^n\to \P^N$ with respect to $(\C^{\times})^{N+1}$-action is called the {\emph{Chow polytope}}.
In Section $\ref{sec:SecPoly}$, we recall that the Chow polytope of a projective toric variety $X_P$ coincides with the secondary polytope 
$\sec(A)$ that is a polytope whose vertices are corresponding to regular triangulations of $P=\conv(A)\subset \R^n$.

Next we define the $X$-discriminant of a linearly normal smooth projective variety $X$ in $\P^N$. Let $\P_N^\vee$ be the dual projective space of
$\P^N$ so that each point of $\P_N^\vee$ corresponds to a hyperplane in $\P^N$.  
We denote the Zariski tangent space to $X$ at $p\in  X$ by $\mathbb T_p(X)$.
\begin{definition}\rm
The {\emph{dual variety}} $X^\vee \subset \P_N^\vee$ of $X\to \P^N$ is 
the set of tangent hyperplanes to $X$ given by
\begin{equation*}
X^\vee :=
\Set{f\in \P_N^\vee |\mathbb T_p(X) \subset \ker(f) \text{ \,for\, some\, } p\in X}.
\end{equation*}
Then we define the {\emph {dual defect}} of $X\subset \P^N$ by
\[
\delta(X):=N-\dim X^\vee -1.
\]
\end{definition}
When $X^\vee$ has codimension one in $\P_N^\vee$ (i.e., the dual defect is zero), there exists an irreducible homogeneous defining polynomial
(unique up to multiple constants) $\D_X$ of $X^\vee$ which is often called the {\emph{$X$-discriminant form}}:
\[
X^\vee=\set{f\in \P_N^\vee | \D_X(f)=0}.
\]
Following notations and terminologies in \cite{Sean12}, we regard $X$-discriminants and $X$-resultants as homogeneous polynomials on spaces
of matrices:
\[
\D_X\in \C[\P_N^\vee]\cong \C[M_{1\times (N+1)}] \qquad \text{and} \qquad R_X\in \C[M_{(n+1)\times (N+1)}].
\]
Then the actions of $\sigma \in \SL(N+1,\C)$ on $R_X$ and $\D_X$ are determined by
\begin{align*}
\sigma \cdot \D_X\left(A\right)&=\D_X\left(A\cdot \sigma\right), \quad A=(a_i)_{0\leqslant i \leqslant N} \in M_{1\times (N+1)} \qquad \text{and}\\
\sigma \cdot R_X\left(B\right)&=R_X\left(B\cdot \sigma \right), \quad B=(b_{ij})_{0\leqslant i \leqslant n,\, 0\leqslant j \leqslant N}\in M_{(n+1)\times (N+1)}
\end{align*}
respectively. 

\subsection{Higher associated hypersurfaces and the Hurwitz forms}\label{sec:Hurwitz}
We generalize the construction of the associated hypersurface $Z_X$ in Section \ref{sec:CD} as follows.
Let $X^n \subset \P^N$ be an $n$-dimensional irreducible 
smooth subvariety of degree $d\geqslant 2$.
For each $0\leqslant k \leqslant n$, we define {\emph{$k$-th associated subvariety}} $\mathcal Z_k(X)$ in $\G(N+k-(n+1),\P^N)$ by
\begin{equation*}
\mathcal Z_k(X):=\Set{L\in \G| L\cap X\neq \emptyset, \quad \dim(L\cap \mathbb T_p(X))\geqslant k \text{ \,for\, some\, } p\in L\cap X},
\end{equation*}
where we abbreviate $\G(N+k-(n+1),\P^N)$ by $\G$.
It is known that ``generally" all of $\mathcal Z_k(X)$ are hypersurfaces in $\G$ (Proposition $2.11$, Chapter 3 in \cite{GKZ94}).
Obviously, $0$-th associated hypersurface $\mathcal Z_0(X)$ is $Z_X$ in Section \ref{sec:CD}.
Thus, for each $0\leqslant k \leqslant n$, we can choose the defining polynomial $\nabla_k(X)$ of $\mathcal Z_k(X)$ up to multiple constants if it is a hypersurface in $\G$.
We set $\nabla_k(X)= 1$ if $\codim \mathcal Z_k(X)\geqslant 2$. For $0\leqslant k \leqslant n$, we readily see that $\nabla_k(X)$ coincide with the Chow form $R_X$ if $k=0$,
and the discriminant form $\D_X$ if $k=n$.
In K\"ahler geometry, the most important object is the defining polynomial $\D_1(X)$ of $\mathcal Z_1(X)$.
Note that the $1$-st associated subvariety has {\emph{always}} codimension one due to Sturmfels (Theorem $1.1$ in \cite{St17}).
Hence, there is a unique non-constant defining polynomial $\Hu_X$ of $\mathcal Z_1(X)$ which is called the {\emph{Hurwitz form}} of $X\subset \P^N$.

On the other hand, for a given projective embedding $X^n\to \P^N$, we can consider the following Segre embedding
\begin{equation}\label{map:Segre}
X\times \P^{n-1}\longrightarrow \P\left( M^\vee_{n\times(N+1)}\right).
\end{equation}
Then it is known that the dual defect $\delta(X\times \P^{n-1})$ of the Segre image of $\eqref{map:Segre}$ is zero due to the result of Weyman and Zelevinsky
(see, Corollary 3.2 in \cite{WZ94}). Hence, there exists a non-constant homogeneous polynomial 
\[
\D_{X\times \P^{n-1}}\in \C[M_{n\times (N+1)}]
\]
such that
\[
(X\times \P^{n-1})^{\vee}=\set{f\in  \P\left( M^\vee_{n\times(N+1)}\right) | \D_{X\times \P^{n-1}}(f)=0}.
\]
In \cite{Sean12}, $\D_{X\times \P^{n-1}}$ is called the {\emph{$X$-hyperdiscriminant}} of $X\to \P^N$.

Now we see the relationship between $\Hu_X$ and $\D_{X\times \P^{n-1}}$.
As explained in \cite[p. $272$]{Sean12}, it is known that $\D_{X\times \P^n}$ coincides with the Chow form $R_X$ by the {\emph{Cayley trick}}.
Analogously, it was proved in \cite{St17} that the Hurwitz form (in the Stiefel coordinates) coincides with the hyperdiscriminant form $\D_{X\times \P^{n-1}}$ (see also, Section 3.2 in \cite{Sean21}).
Thus, we can state the following definition by adapting Definition \ref{def:Pss} into the pair of $R_X$ and $\Hu_X$.
\begin{definition}\label{def:Pss2}\rm
Let $X\to \P^N$ be a linearly normal irreducible projective variety with degree $d\geqslant 2$. Then there is the natural action of $G=\SL(N+1,\C)$ on $\P^{N}$
(which induces the action of $G$ on $X$). $X$ is said to be {\emph{numerically semistable}} if the pair of the invariant polynomials $\left(R_X^{\deg \Hu_X}, \Hu_X^{\deg R_X} \right)$ is numerically semistable in the sense of
Definition \ref{def:Pss} with respect to $G=\SL(N+1,\C)$-action.
\end{definition}
For a fixed positive integer $i>0$, we shall consider the Kodaira embedding $\Psi_i:X\to \P(H^0(\mathcal O_X(i))^*)=\P^{N_i}$.
Let $R_X(i)$ (resp. $\Hu_X(i)$) be the associated Chow form (resp. Hurwits form) of $\Psi_i(X)$.
Then we have the natural $\SL(N_i+1, \C)$-action on $\left(R_X(i)^{\deg \Hu_X(i)}, \Hu_X(i)^{\deg R_X(i)} \right)$
\begin{definition}\rm
A polarized variety $(X,L)$ is {\emph{asymptotically numerically semistable}} if $(X,L^{\otimes i})$ is numerically semistable for any sufficiently large integer $i$.
\end{definition}
In \cite[Theorem A]{Sean12}, Paul proved that the $K$-energy restricted to the Bergman potentials can be written as the difference between the log norms of the Chow form and the Hurwitz form.
He also proved that for $X\subset \P^N$ and the Fubini-Study form $\omega_{\rm{FS}}$ of $\P^N$, the $K$-energy of $(X,\restrict{\omega_{\rm{FS}}}{X})$ is bounded from below if and only if it is numerically semistable
(Theorem C in \cite{Sean12}).

In the following section, we further develop the concept of numerical semistability in the case where $X\to \P^N$ is a projective {\emph{toric}} variety building upon works of Gelfand-Kapranov-Zelevinsky ($A$-Resultants and $A$-Discriminants).

\section{$A$-Resultants and $A$-Discriminants}\label{sec:GKZ}
This section collects some fundamental properties and a combinatorial description of the Newton polytope of $A$-discriminants.
We refer the reader to \cite{GKZ94, LRS10, KL20} and references therein for further details.
Throughout this section, we always assume that a finite set of points $A=\set{\bs a_0, \dots, \bs a_N}$ in $\Z^n$ satisfies $(*)$.

\subsection{A construction of the secondary polytopes}\label{sec:SecPoly}
We firstly recall the definition of the secondary polytopes and their properties.
A {\emph{triangulation}} of $(P,A)$ is a collection $T$ of $n$-simplices all of whose vertices are points in $A$ that satisfies the 
following three propeties:
\begin{enumerate}
\item The union of all these simplices equals $P$; $\bigcup_{\sigma \in T}\sigma =P$.
\item $T$ is closed under taking faces; if $\sigma \in T$ and $\sigma'\preceq \sigma$, then $\sigma'\in T$.
\item Any pair of these simplices intersects in a common face (possibly empty). 
\end{enumerate}
In the above, $F'\preceq F$ denotes $F'$ is a face of $F$.
Recall that a triangulation $T$ of $(P,A)$ is called {\emph{regular}} if it can be obtained by projecting all the lower faces of a lifted finite set $A^\omega$ defined by
\[
A^\omega:=\Set{\widehat{\bs a}_0, \dots , \widehat{\bs a}_N}\subset \R^{n+1},\qquad  \widehat{\bs a}_i=
\begin{pmatrix}
\bs a_i \\ \omega_i
\end{pmatrix},
\]
for some $\omega=(\omega_0, \dots , \omega_N)\in \R^{N+1}$. Then, a vector $\omega\in \R^{N+1}$ can be thought of as a function
$\omega:A\to \R$ with $\omega(\bs a_i)=\omega_i$, which is often called a {\emph{height function}}.
See, p. $756$ in \cite{Yotsu16} for more detailed descriptions.
Let $T$ be a (not necessary regular) triangulation of $(P,A)$ and let $J=\set{0, \dots , N}$ be the index set of labels.
Let $\vol(\cdot)$ denote a translation invariant volume form on $\R^n$ with the normalization $\vol(\D_n)=1/n!$ for the standard $n$-dimensional simplex $\D_n=\conv\set{\bs e_i| 1\leqslant i \leqslant n}$, where $\bs e_i$ denotes the standard basis of $\R^n$.
For any simplex $\sigma$ of $T$, we denote the set of vertices of $\sigma$ by $\mathcal V(\sigma)$. Fix a point $\bs a_j\in A$.
Then we consider the function $\phi_{A,T}:A\to \R$ defined by
\begin{equation*}\label{eq:PreGKZ}
\phi_{A,T}(\bs a_j)=\sum_{\sigma:\bs a_j\in \mathcal V(\sigma)}n!\vol(\sigma),
\end{equation*}
where the summation is over all maximal simplices of $T$ for which $\bs a_j$ is a vertex.
In particular, $\phi_{A,T}(\bs a_j)=0$ for $j\in J$ if and only if $\bs a_j\in A$ is not a vertex of any simplex of $T$.
For each triangulation $T$ of $(P,A)$, we define the {\emph{Gel'fand-Kapranov-Zelevinsky (GKZ) vector}} of $T$ in $\Z^A\cong \Z^{N+1}$ by
\begin{equation}\label{eq:GKZvec}
\Phi_A(T)=\sum_{j\in J}\phi_{A,T}(\bs a_j)\bs e_j ,
\end{equation}
where $\bs e_j$ for $j\in J$ is the standard basis of $\R^{N+1}$. The {\emph{secondary polytope}} $\sec(P)$ is the convex polytope
in $\R^{N+1}$ defined by
\begin{equation}\label{def:SecPoly}
\sec(P):=\conv\set{\Phi_A(T)| T\text{ is a triangulation of } (P,A)}.
\end{equation}
It is known that the secondary polytopes have the following properties. 
\begin{theorem}[\cite{GKZ94}, p. $221$, Theorem $1.7$]\label{thm:GKZ}
Let $A=\set{\bs a_0, \dots, \bs a_N}$ be a finite set of points satisfying $(*)$. Then we have
\begin{enumerate}
\item $\dim \sec(P)=N-n$.
\item The vertices of $\sec(P)$ are bijectively corresponding to the regular triangulations of $(P,A)$.
If $T$ is a regular triangulation of $(P,A)$, then $\Phi_A(T)\neq \Phi_A(T')$ for any other triangulation $T'$ of $(P,A)$.
\end{enumerate}
\end{theorem}
In order to see the relationship between the secondary polytopes and the Chow polytopes of projective toric varieties,
we next review on the construction of toric varieties. See, Chapter $5$ in \cite{GKZ94} for further details.

A {\emph{toric variety}} is a complex irreducible algebraic variety with a complex torus action having an open dense orbit.
As usual, let $A=\set{\bs a_0, \dots, \bs a_N}$ be a finite set of integer vectors in $\Z^n$ which satisfies $(*)$ and let $P$ be its
convex hull. Setting
\[
X_P^0=\set{[\bs x^{\bs a_0}:\cdots :\bs x^{\bs a_N}]\in \P^N|\bs x=(x_1,\dots ,x_n)\in (\C^\times)^n},
\]
we define the variety $X_P\subset \P^N$ to be the Zariski closure of $X_P^0$ in $\P^N$. 
Then $X_P$ is an $n$-dimensional equivariantly embedded subvariety in $\P^N$. We call $X_P\subset \P^N$ the toric variety
associated with $(P,A)$.

Now we consider the Chow point $R_{X_P}$ of the toric variety $X_P$. Let $\Ch(P)$ denote the Chow polytope of $X_P$
which is the weight polytope $\Wt_{(\C^\times)^{N+1}}(R_{X_P})$ of $R_{X_P}$ with respect to the algebraic torus action 
$(\C^\times)^{N+1}$. It was proved in \cite{KSZ92} that the Chow polytope $\Ch(P)$ coincides with the secondary polytope.
Combining this and Proposition $1.11$ in \cite[p. $222$]{GKZ94}, we have the following combinatorial description of the affine hull
of the Chow polytope of $X_P$: 
\begin{equation}\label{eq:AffChow}
\mathrm{Aff}_\R(\Ch(P))=\Set{(\psi_0, \ldots ,\psi_N)\in \R^{N+1}|
\begin{matrix}
\displaystyle \sum_{i=0}^N\psi_i=(n+1)!\vol(P) \\
\displaystyle \sum_{i=0}^N\psi_i\bs a_i=(n+1)!\int_P\bs x \,dv 
\end{matrix}
}.
\end{equation}

Next we define a $T$-piecewise linear function $g_{\varphi,T}:P\to \R$ as follows.
Let $T$ be a triangulation of $(P,A)$ and let $\varphi=(\varphi_0, \dots , \varphi_N)\in \R^{N+1}$ be a height function with 
$\varphi(\bs a_i)=\varphi_i$. Then, there is a unique $T$-piecewise-linear function $g_{\varphi, T}(\bs a_i)=\varphi_i$ and extended 
affinely on $\sigma$ for each maximal simplex $\sigma$ of $T$.
Remark that in the definition of $g_{\varphi,T}$, we do not require $\varphi$ to be the height function that induces the triangulation $T$.

For later use, we state the following equality between the GKZ vector and the integration of a $T$-piecewise-linear function $g_{\varphi,T}$
over the polytope $P$. 
\begin{lemma}
Let $\varphi \in \R^{N+1}$ be a height function and let $T$ be any triangulation of $(P,A)$.
For any simplex $\sigma\in T$, we identify a point $\bs a_j \in \sigma$ with its index $j\in J=\set{0,\dots ,N}$.
Then, we have
\begin{align*}
\braket{\varphi, \Phi_A(T)}&=(n+1)!\int_P g_{\varphi,T}(\bs x)\, dv \\
&=n!\sum_{\sigma\in T}\vol(\sigma)\sum_{j\in \sigma}\varphi_j.
\end{align*}
\end{lemma}
\begin{proof}
See, Proposition $3.5$ and Lemma $3.6$ in \cite{Yotsu16}.
\end{proof}

\subsection{Massive GKZ vectors and the discriminant polytopes}\label{sec:MassiveGKZ}
Let $T$ be a triangulation of $(P,A)$. For $0\leqslant j\leqslant n$, a $j$-dimensional simplex $\sigma$ of $T$ is said to be {\emph{massive}} if it is contained in some $j$-dimensional face of $P$.
In that case, the face containing $\sigma$ is unique and we denote it by $\Gamma(\sigma)$.
For any $n$-face $\sigma$ of $T$, we see that $\sigma$ is massive with $\Gamma(\sigma)=P$.

Let $\Gamma \preceq P$ be a face, and let $\mathrm{Aff}_\R(\Gamma)\subset\R^n$ be its affine hull.
We introduce the lattice volume form $\vol_{\Z(\Gamma)}(\cdot)$ (or $\vol_\Z(\cdot)$) on $\mathrm{Aff}_\R$, induced by the lattice 
$\mathrm{Aff}_\Z(A\cap \Gamma)$. Note that the volume form $\vol_\Z(\cdot)$ is normalized so that $\vol_\Z(\D_n)=1$ for the standard
$n$-simplex $\D_n$. For $\bs a_i\in A$, let $M_T^j(\bs a_i)$ be the set of massive $j$-simplices of $T$ that have $\bs a_i$ as a vertex.
Then, we obtain functions $\eta_{T,j}:A\to \Z$ by setting
\begin{equation*}
\eta_{T,j}(\bs a_i)=\sum_{\sigma\in M_T^j(\bs a_i)}\vol_\Z(\sigma).
\end{equation*}
For the index set $I=\set{0,1,\dots , N}$ of labels, we consider the vector $\eta_{T,j}$ defined by
\begin{equation}\label{eq:MasGKZ}
\eta_{T,j}=\sum_{i\in I}\eta_{T,j}(\bs a_i)\bs e_i\in \Z^A\cong \Z^{N+1}, \qquad\text{for}\quad j=0,\dots, n.
\end{equation}
Remark that $\eta_{T,n}$ coincides with the GKZ vector defined in $\eqref{eq:GKZvec}$, because every $n$-face $\sigma$ of $T$ is massive.
In fact, we see
\begin{align*}
\eta_{T,n}&=\sum_{i\in I}\eta_{T,n}(\bs{a}_i)\bs e_{i}=\sum_{\sigma \in M_{T}^n(\bs a_i)}\vol_\Z(\sigma)\bs e_i
=\sum_{\sigma\in T \atop \bs a_i\in \sigma} n!\cdot\vol(\sigma)\bs e_i=\Phi_A(T).
\end{align*}
\begin{definition}\rm
For a triangulation $T$ of $(P,A)$, we define the {\emph{massive GKZ vector}} $\eta_T\in \Z^A$ of $T$ by
\[
\eta_T=\sum_{j=0}^n(-1)^{n-j}\eta_{T,j}.
\]
Moreover, two triangulations $T$ and $T'$ of $(P,A)$ are called {\emph{$D$-equivalent}} if $\eta_T=\eta_{T'}$.
\end{definition}
For a projective toric variety $X_P \to \P^N$, we denote the $X$-discriminant by $\D_{X_P}$.
In the literature of GKZ theory, $\D_{X_P}$ (resp. $R_{X_P}$) is called the {\emph{$A$-discriminant}} (resp. {\emph{$A$-resultant}}).
Let $\Wt_{(\C^{\times})^{N+1}}(\D_{X_P})$ be the weight polytope of the $A$-discriminant of $X_P\to \P^N$ with respect to $(\C^\times)^{N+1}$-action. Then, we have the following.
\begin{theorem}[Chapter $11$, Theorem $3.2$ in \cite{GKZ94}]
The weight polytope $\Wt_{(\C^{\times})^{N+1}}(\D_{X_P})$ coincides with the convex polytope $\D(P)$ in $\R^{N+1}$ defined by
\begin{equation}\label{def:discrim}
\D(P):=\conv\set{\eta_T| \text{$T$ is a triangulation of $(P,A)$}}.
\end{equation}
In particular, the vertices of $\D(P)$ are exactly the points $\eta_T$ for all regular triangulations $T$ of $(P,A)$.
More precisely, the set of vertices $\mathcal V(\D(P))$ is bijectively corresponding to the $D$-equivalence classes of regular triangulations of $(P,A)$.
\end{theorem}
We call $\D(P)$ the {\emph{discriminant polytope}} of $X_P\subset \P^N$.
It is known that the secondary polytope $\sec(P)$ is given by a Minkowski sum of the discriminant polytope and another polytope $D(P)$:
\begin{equation}\label{eq:MinSum}
\sec(P)=\D(P)+D(P).
\end{equation}
Conversely, $D(P)$ can be uniquely determined by $\eqref{eq:MinSum}$.
See, Remark $3.4$ (b) in \cite[p. $362$]{GKZ94} for more detailed descriptions.

Then, we have the following description of the affine hull of the discriminant polytope of $X_P$ analogous to the affine hull of the Chow polytope of $X_P$ $\eqref{eq:AffChow}$. 
\begin{proposition}\label{prop:AffDiscrim}
Let $A=\set{\bs a_0, \ldots , \bs a_N} \subset M_\Z$ be a finite set of integer vectors which satisfies $(*)$.
Let $P$ denote the convex hull of $A$ in $M_\R$. Then the affine hull of the discriminant polytope $\D(P)$ in $\R^{N+1}$ is equal to
\begin{align}\label{eq:AffDiscrim}
\begin{split}
&\mathrm{Aff}_\R \bigl(\D(P)\bigr) \\
&=\Set{(\eta_0, \dots ,\eta_N)\in \R^{N+1}|
\begin{matrix}
\displaystyle \sum_{i=0}^N\eta_i=\sum_{F\preceq P}(-1)^{\mathrm{codim}\, F}(\dim F+1)!\vol(F)  \hspace{0.9cm} \color{white}{2} \\
\displaystyle \sum_{i=0}^N\eta_i\bs a_i=\sum_{j=0}^n\sum_{F\in \mathcal F^j(P)}(-1)^{\mathrm{codim}\, F}(\dim F+1)!\int_F\bs x \,d\sigma_F 
\end{matrix}
}.
\end{split}
\end{align}
\end{proposition}
\begin{proof}
Based on the argument of the proof of $\eqref{eq:AffChow}$, we prove $\eqref{eq:AffDiscrim}$.
In fact, since 
\[
\dim\left( \Wt_{(\C^\times)^{N+1}}(\D_{X_P})\right)=\dim \Ch(P)=N-n,
\]
the equalities
\begin{align}
\sum_{i=0}^N\eta_i&=\sum_{F\preceq P}(-1)^{\mathrm{codim}\, F}(\dim F+1)!\vol(F), \label{eq:AffDiscrim1} \\
\sum_{i=0}^N\eta_i\bs a_i&=\sum_{j=0}^n\sum_{F\in \mathcal F^j(P)}(-1)^{\mathrm{codim}\, F}(\dim F+1)!\int_F\bs x \,d\sigma_F  \label{eq:AffDiscrim2}
\end{align}
define the same dimensional affine subspace in $\R^{N+1}$. Hence, it suffices to show that the equalities $\eqref{eq:AffDiscrim1}$ and $\eqref{eq:AffDiscrim2}$ hold on $\D(P)$. 

Obviously, $\eqref{eq:AffDiscrim1}$ holds on 
$\D(P)$ by the definition of the discriminant polytope $\eqref{def:discrim}$.
In particular, $\eqref{eq:AffDiscrim1}$ corresponds to the degree of $\D_{X_P}$ (Theorem $2.8$ in \cite[Chapter $9$]{GKZ94}).

In order to prove the second assertion, we use the following lemma.
\begin{lemma}\label{lem:alternate}
For an $n$-dimensional Delzant polytope $P$, let $\p^j P$ denote the union of codimension $j$-faces of $P$, i.e., 
$\p^jP=\bigcup_{F\in\mathcal F^j(P)}F$ where $\mathcal F^j(P)=\set{F\preceq P| \codim F=j}$. 
Let us denote the natural $(n-j)$-dimensional Lebesgue measure of $\p^j P$ by $d\sigma^{(j)}$.
For every triangulation $T$ of $(P,A)$, let $\eta_T$ be the Massive GKZ vector $\eta_T\in \Z^A$.
Then, for any $\varphi\in \R^A\cong \R^{N+1}$, we have
\begin{equation*}
\braket{\varphi,\eta_T}=\sum_{j=0}^n(-1)^j(n+1-j)!\int_{\p^jP}g_{\varphi,T}(\bs x)\,d\sigma^{(j)}.
\end{equation*}
\end{lemma}
Now we shall identify $\R^A$ with its dual space by the scalar product
\[
\braket{\varphi, \eta_T}=\sum_{\bs a\in A}\varphi(\bs a)\eta_T(\bs a).
\]
Then, Lemma \ref{lem:alternate} gives the equality
\begin{equation}\label{eq:alternate}
\sum_{\bs a\in A}\vp(\bs a)\eta_T(\bs a)=\sum_{j=0}^n(-1)^j(n+1-j)!\int_{\p^{j} P}g_{\vp,T}(\bs x)d\sigma^{(j)}.
\end{equation}
Let us consider the vector-valued linear function $\vp:P\to \R^n$, $\bs x \mapsto \bs x$ on each component.
Applying $\eqref{eq:alternate}$ to this linear function $\varphi$, we have
\begin{align*}
&\sum_{\bs a\in A}\vp(\bs a)\eta(\bs a)=\sum_{j=0}^n(-1)^j(n+1-j)!\int_{\p^j P}\vp(\bs x)d\sigma^{(j)} \\
\Leftrightarrow \quad  &\sum_{\bs a\in A}\eta(\bs a) \bs a=\sum_{j=0}^n(-1)^j(n+1-j)!\int_{\p^j P}\bs x\, d\sigma^{(j)} \qquad (\because ~ \vp(\bs x)=\bs x)\\
\Leftrightarrow \quad  &\sum_{\bs a\in A}\eta(\bs a) \bs a=\sum_{j=0}^n\sum_{F\in \mathcal F^j(P)}(-1)^{\mathrm{codim}\, F}(\dim F+1)!\int_{F}\bs x\, d\sigma_F.
\end{align*}
Hence, the assertion is verified.
\end{proof}
\begin{proof}[Proof of Lemma \ref{lem:alternate}]
This follows from the definition of the massive GKZ vector $\eta_T$ and the $T$-piecewise-linear function $g_{\varphi,T}$. We shall also use the fact that the integral of a linear function on a domain is equal to
the multiplication of the volume of a domain with the value of a linear function at the barycenter (cf. \cite[p. $758$]{Yotsu16}). 
In particular, we have
\begin{equation}\label{eq:lin_bary}
\int_C g_{\varphi, T}(\bs x)\, dv=\vol_\Z(C)\cdot g_{\varphi, T}(\bs b_C),
\end{equation}
for an $n$-dimensional simplex $C$. Here $\bs b_C$ denotes the barycenter of $C$.

Consequently, it suffices to show the equality 
\begin{equation}\label{eq:alternate2}
(n+1-j)!\int_{\p^jP}g_{\varphi, T}(\bs x)d\sigma^{(j)}=\braket{\varphi, \eta_{T,n-j}}\qquad j=0,1, \dots, n
\end{equation}
for proving the assertion. We inductively show $\eqref{eq:alternate2}$.

For $j=0$, this is Lemma $1.8$ in \cite[Chapter $7$]{GKZ94} because $\eta_{T,n}$ coincides with the GKZ vector $\Phi_A(T)$.

For $j=1$, we consider an $(n-1)$-dimensional simplex $C$ in $\p P$. Then, $\eqref{eq:lin_bary}$ implies
\begin{equation}\label{eq:lin_bary2}
\int_C g_{\varphi, T}(\bs x)\, d\sigma=\vol_\Z(C)\cdot g_{\varphi, T}(\bs b_C).
\end{equation}
Since the barycenter of a simplex is given by the average of its vertices, we find
\begin{align*}
g_{\varphi, T}(\bs b_C)=\sum_{\bs a_j\in \mathcal V(C)}\frac{g_{\varphi, T}(\bs a_j)}{n!}=\frac{1}{n!}\sum_{j\in C}\varphi_j.
\end{align*}
Plugging this into $\eqref{eq:lin_bary2}$, we have
\begin{equation}\label{eq:lin_bary3}
\int_C g_{\varphi, T}(\bs x)\, d\sigma=\frac{\vol_\Z(C)}{n!}\sum_{j\in C} \varphi_j.
\end{equation}
Now we show the equality $\eqref{eq:alternate2}$ from the left hand side to the right hand side.
In fact, $\eqref{eq:lin_bary3}$ implies that
\begin{align*}
n!\int_{\p P}g_{\varphi, T}(\bs x)\, d\sigma&=n!\sum_{C\in M_T^{n-1}}\int_C g_{\varphi, T}\, d\sigma=n!\sum_{C\in M_T^{n-1}}\frac{\vol_\Z(C)}{n!}\sum_{j\in C}\varphi_j\\
&=\sum_{C\in M_T^{n-1}}\sum_{j\in C}\varphi_j \vol_\Z(C)=\sum_{j\in J}\sum_{C\in M_T^{n-1}(\bs a_j)\atop j\in C}\varphi_j \vol_\Z(C)\bs e_j\\
&= \sum_{j\in J}\varphi_j(\eta_{T,n-1})\bs e_j=\braket{\varphi, \eta_{T, n-1}}.
\end{align*}
By the inductive argument, we obtain the desired equality $\eqref{eq:alternate2}$ for $j=2, 3, \dots, n$.
The assertion is verified.
\end{proof}

\subsection{Hurwitz vectors and Hurwitz polytopes}\label{sec:combinat}
In this section, we firstly recall the definition of the {\emph{Hurwitz vector}} as an analogue of the GKZ vector defined in \eqref{eq:GKZvec}.
Recently, it has been proved by Sano that the weight polytope of the Hurwitz form of  a smooth polarized toric variety $(X_P,L_P)$ coincides with the convex hull of the Hurwitz vectors for any triangulation
of $(P,A)$ \cite[Theorem $1.4$]{Sn23}. This can be interpreted as an analogue of the secondary polytope defined in $\eqref{def:SecPoly}$ (see, also Theorem \ref{thm:GKZ} (2)). 

Later on in Lemma \ref{lem:degHu}, we will compute the degree of the Hurwitz form of a smooth projective toric variety $X_P\subset \P^N$ in terms of the associated moment polytope $P\subset M_\Z$.
For the proof, we use the technique of $A$-discriminants in \cite[Chapter $9$]{GKZ94} and the binomial theorem (Claim \ref{claim:binomial}) which was firstly found in probability theory \cite{Fl67}.
For reader's convenience, we provide a proof of Claim \ref{claim:binomial} using Vandermonde's identity. 
One can see that the technique in the proof of Lemma \ref{lem:degHu} plays a key role in Section \ref{sec:KeyProps}. 

Finally in Proposition \ref{prop:HMPcriterion}, we see that the Paul's criterion that is a necessary and sufficient condition for a smooth polarized (toric) variety $(X, L^{\otimes i})$ to be numerically semistable.

Let $T$ be a triangulation of an $n$-dimensional polytope $(P,A)$ as in Section $\ref{sec:MassiveGKZ}$.
Let $\eta_{T,n}$ and $\eta_{T,n-1}$ be the vectors defined in $\eqref{eq:MasGKZ}$. Following Definition $1.3$ in \cite{Sn23}, we define the {\emph{Hurwitz vector}} by
\begin{equation}\label{def:HurVec}
 \xi_T:=n\eta_{T,n}-\eta_{T,n-1}.
\end{equation}
For the associated smooth projective toric variety $X_P\subset \P^N$, we define the {\emph{Hurwitz polytope}} by
\[
\Hu(P):=\Wt_{(\C^\times)^{N+1}}(\Hu_{X_P}).
\]
The following combinatorial description of $\Hu(P)$ has been found in \cite{Sn23}.
\begin{proposition}[Sano]\label{prop:HuPoly}
Let $\Hu(P)$ be the Hurwitz polytope of $P$ corresponding to a smooth projective toric variety $X_P\subset \P^N$.
Then, we have the equality in $\R^{N+1}$ such that
\[
\Hu(P)=\conv\set{\xi_T| T \text{ is a triangulation of } P},
\]where $\xi_T$ is the Hurwitz vector with respect to $T$.
\end{proposition}
As we already mentioned in Section \ref{sec:Hurwitz}, the hyperdiscriminant form (resp. hyperdiscriminant polytope) coincides with the Hurwitz form (resp. Hurwitz polytope).
Hence, numerical semistability of a polarized smooth toric variety can be reduced to the combinatorial description of the pair of $(\sec(P), \Hu(P))$.

The following lemma is the first key step to prove our main theorem (Theorem \ref{thm:main}) in this paper.
\begin{lemma}\label{lem:degHu}
Let $X_P^n\subset \P^N$ be the smooth projective toric variety associated with an $n$-dimensional Delzant polytope $(P,A)$, where $A=P\cap \Z^n=\set{\bs a_0, \dots, \bs a_N}$ is the set of the lattice points in $P$.
Let $\Hu_{X_P}$ be the Hurwitz form of $X_P$. Then the degree of $\Hu_{X_P}$ equals 
\[
n(n+1)!\vol(P)-n!\vol(\p P).
\]
\end{lemma}
\begin{proof}
Let $\D_{X_P}$ be the $A$-discriminant of $(P,A)$. 
For the proof, we shall use Theorem $2.8$ in \cite[Chapter $9$]{GKZ94}, which states that
\begin{equation}\label{eq:degDiscrim}
\deg \D_{X_P}=\sum_{F\preceq P}(-1)^{\codim F}(\dim F+1)!\vol(F),
\end{equation}
where $F$ runs over all nonempty faces of $P$.

Let $\D_{n-1}$ be the $(n-1)$-dimensional standard simplex with $\vol(\D_{n-1})=\frac{1}{(n-1)!}$.
Let $Q=P\times \D_{n-1}$ be the products polytope of $P$ and $\D_{n-1}$ which is the $(2n-1)$-dimensional lattice polytope.
Since any nonempty face $F$ of $Q$ is obtained by the product of a nonempty face $F_1\preceq P$ and a nonempty face $F_2\preceq \D_{n-1}$,
$F$ can be written in the form of $F=F_1\times F_2$.
Thus, we can apply a similar argument in \cite{OSY23, Yotsu23} for the faces of $Q$ as follows: for a fixed $j$ with $0\leqslant j \leqslant 2n-1$,
let $\p^jQ$ be the union of codimension $j$-faces of $Q$ as in Lemma \ref{lem:alternate}. In particular, we remark that $\p^0Q=Q$ and $\p^{2n-1}Q=\mathcal V(Q)$.
We also use the same terminologies for $\p^i P$ and $\p^k\D_{n-1}$ respectively. Then, we readily see that
\begin{equation}\label{eq:splitting}
\p^jQ=\smallcup_{i+k=j}\p^i P\times \p^k\D_{n-1}, 
\end{equation}
where $i$ and $j$ run over all $0\leqslant i \leqslant n$ and $0\leqslant k \leqslant n-1$ satifying $i+k=j$.
Especially, the equalities
\[
\p Q=(\p P\times \D_{n-1})\cup(P\times \p\D_{n-1}) \quad \text{and} \quad \mathcal V(Q)=\mathcal V(P)\times \mathcal V(\D_{n-1})
\]
hold. Hence, the straightforward computations with $\eqref{eq:splitting}$ give that
\begin{align}
\begin{split}\label{eq:VolQ}
\vol(\p^jQ)&=\sum_{i+k=j}\vol(\p^i P)\cdot \vol(\p^k \D_{n-1})\\
&=\sum_{i+k=j} {{n}\choose{n-k}}\frac{1}{(n-k-1)!}\vol(\p^i P),
\end{split}
\end{align} 
where we used $\vol(\p^k \D_{n-1})= {{n}\choose{n-k}}\frac{1}{(n-k-1)!}$ in the last equality.

Now we regard the associated smooth projective toric variety $Y_Q$ as $X_P\times \P^{n-1}$. Let us denote $Q\cap \Z^{2n+1}$ by $\widetilde A$.
Applying $\eqref{eq:degDiscrim}$ into the $\widetilde A$-discriminant of $(Q,\widetilde A)$, we calculate the degree of $\Hu_{X_P}$ such as
\begin{align}\label{eq:degHu2}
\begin{split}
\deg(\Hu_{X_P})&=\deg(\D_{X\times \P^{n-1}})\\
&=\deg(\D_{Y_Q})=\sum_{j=0}^{2n-1}(-1)^j(2n-j)\vol(\p^j Q).
\end{split}
\end{align}
Plugging $\eqref{eq:VolQ}$ into $\eqref{eq:degHu2}$, we can compute $\deg\Hu_{X_P}$ in terms of the moment polytope $P$, instead of calculating it in terms of $Q$.
By simplifying each coefficient of $\vol(\p^i P)$ in
\[
\sum_{j=0}^{2n-1}(-1)^j(2n-j)\sum_{i+k=j}{{n}\choose{n-k}}\frac{1}{(n-k-1)!}\vol(\p^iP),
\]
we find the following:

\vskip 4pt

\noindent (I) $\mathrm{Coeff}(\vol(P))$. 
\begin{align*}
\sum_{k=1}^n(-1)^{n-k}\frac{(n+k)!}{(k-1)!}{{n}\choose{k}}&=\sum_{k=1}^n(-1)^{n-k}\cdot (n+1)! \cdot \frac{(n+k)!}{(k-1)!(n+1)!}{{n}\choose{k}}\\
&=(n+1)!\sum_{k=1}^n(-1)^{n-k}{{n}\choose{k}}{{n+k}\choose{k-1}}.
\end{align*} 

\noindent (II) $\mathrm{Coeff}(\vol(\p P))$. 
\begin{align*}
-\sum_{k=1}^n(-1)^{n-k}\frac{(n+k-1)!}{(k-1)!}{{n}\choose{k}}&=-\sum_{k=1}^n(-1)^{n-k}\cdot n! \cdot \frac{(n+k-1)!}{(k-1)!n!}{{n}\choose{k}}\\
&=-n!\sum_{k=1}^n(-1)^{n-k}{{n}\choose{k}}{{n+k-1}\choose{k-1}}.
\end{align*} 

\noindent (III) $\mathrm{Coeff}(\vol(\p^i P))$ for $2\leqslant i \leqslant n$. 
\begin{align*}
\sum_{k=1}^n(-1)^{n-k}\frac{(n+k-i)!}{(k-1)!}{{n}\choose{k}}&=\sum_{k=1}^n(-1)^{n-k}\cdot (n-i+1)! \cdot \frac{(n+k-i)!}{(k-1)!(n-i+1)!}{{n}\choose{k}}\\
&=(n-i+1)!\sum_{k=1}^n(-1)^{n-k}{{n}\choose{k}}{{n+k-i}\choose{k-1}}.
\end{align*} 
We complete the proof by the following beautiful equalities of combinatorics used in probability theory.
The following claim can be seen as a slight generalization of the original one (cf. $(12.13)$ in \cite[p. $64$]{Fl67}). 
\begin{claim}\label{claim:binomial}
Let $m$ be a fixed positive integer. For any integer $0\leqslant n \leqslant m$, we have the equality
\begin{equation}\label{eq:binomial}
\sum_{k=1}^n(-1)^{n-k}{{n}\choose{k}}{{m+k}\choose{m+1}}={{m}\choose{n-1}}.
\end{equation}
\end{claim}
\begin{proof}[Proof of Claim \ref{claim:binomial}]
Firstly, we recall that for any integer $a\in \Z$, we have the identity
\begin{equation}\label{eq:binomial0}
{{-a}\choose{k}}=(-1)^k{{a+k-1}\choose{k}},
\end{equation}
as seen in \cite[p. $63$, $(12.4)$]{Fl67}. This yields that
\begin{equation}\label{eq:binomial2}
{{-m-2}\choose{k-1}}=(-1)^{k-1}{{m+k}\choose{k-1}}.
\end{equation}
Then, we can compute the left hand side of $\eqref{eq:binomial}$ such as
\begin{align}
\sum_{k=1}^n(-1)^{n-k}{{n}\choose{k}}{{m+k}\choose{m+1}}&=\sum_{k=1}^n(-1)^{n-k}{{n}\choose{k}}{{m+k}\choose{k-1}} \notag \\
&=\sum_{k=1}^n(-1)^{n-1}{{n}\choose{k}}{{-m-2}\choose{k-1}} \qquad (\because~ \eqref{eq:binomial2} ) \notag  \\
&=(-1)^{n-1}\sum_{k=1}^n{{n}\choose{n-k}}{{-m-2}\choose{k-1}}  \notag  \\
&=(-1)^{n-1}\sum_{k=0}^{n-1}{{n}\choose{n-1-k}}{{-m-2}\choose{k}}. \label{eq:binomial3}
\end{align}
Using Vandermonde's identity
\[
{{a+b}\choose{r}}=\sum_{k=0}^r{{a}\choose{k}}{{b}\choose{r-k}}
\]
for $r=n-1$, $a=-m-2$ and $b=n$, we obtain that
\begin{equation}\label{eq:binomial4}
\eqref{eq:binomial3}=(-1)^{n-1}{{n-m-2}\choose{n-1}}={{m}\choose{n-1}}.
\end{equation}
In the last equality in $\eqref{eq:binomial4}$, we again used $\eqref{eq:binomial0}$ for $a=m+2-n$ and $k=n-1$.
This completes the proof of Claim \ref{claim:binomial}.
\end{proof}
Now we finish the proof Lemma \ref{lem:degHu} as follows:

\vskip 4pt

\noindent (I) For $\mathrm{Coeff}(\vol(P))$, we use Claim \ref{claim:binomial} for $m=n$. Then, we find
\begin{align*}
\mathrm{Coeff}(\vol(P))&=(n+1)!\sum_{k=1}^n(-1)^{n-k}{{n}\choose{k}}{{n+k}\choose{k-1}}\\
&=(n+1)!\sum_{k=1}^n(-1)^{n-k}{{n}\choose{k}}{{n+k}\choose{n+1}}\\
&=(n+1)!{{n}\choose{n-1}}=(n+1)!\cdot n.
\end{align*}

\noindent (II) Similar to Case (I), we use Claim \ref{claim:binomial} for $m=n-1$. This yields that
 \begin{align*}
\mathrm{Coeff}(\vol(\p P))&=-n!\sum_{k=1}^n(-1)^{n-k}{{n}\choose{k}}{{n+k-1}\choose{k-1}}\\
&=-n!\sum_{k=1}^n(-1)^{n-k}{{n}\choose{k}}{{n-1+k}\choose{n}}\\
&=-n!{{n-1}\choose{n-1}}=-(n!).
\end{align*}

\noindent (III) Finally, we use Claim \ref{claim:binomial} for $m=n-i$ with $2\leqslant i \leqslant n$.
Note that ${{a}\choose{b}}=0$ for $0\leqslant a <b$. Thus, we see that
 \begin{align*}
\mathrm{Coeff}(\vol(\p^i P))&=(n+1-i)!\sum_{k=1}^n(-1)^{n-k}{{n}\choose{k}}{{n+k-i}\choose{k-1}}\\
&=(n+1-i)!\sum_{k=1}^n(-1)^{n-k}{{n}\choose{k}}{{n+k-i}\choose{n-i+1}}\\
&=(n+1-i)!{{n-i}\choose{n-1}}=0.
\end{align*}
The assertion is verified.
\end{proof}

\subsection{Paul's criterion for numerical semistability}\label{sec:HMPcriteria}
We finish this section with the Hilbert-Mumford-Paul (HMP) criterion for numerical semistability of a polarized toric variety.
For this, we recall the {\em{Ehrhart polynomial}} $E_P(t)$ of $P$, which counts the number of lattice points in the $i$-th dilation of an $n$-dimensional lattice Delzant polytope $P$ for $i\in \N$:
\[
E_P(i)=\sharp (P\cap(\Z/i)^n).
\] 
It is well known that the Ehrhart polynomial is a degree $n$ polynomial and can be written in the form of 
\[
E_P(t)=\vol(P)t^n+\frac{\vol(\p P)}{2}t^{n-1}+\dots +1.
\]
As we have seen in Section \ref{sec:Pss}, we then naturally consider $G$-representation for the complex torus $G=(\C^\times)^{E_P(i)}$.
Moreover, the character group of $(\C^\times)^{E_P(i)}$ can be identified with the lattice $\Set{P\cap(\Z/i)^n \to \Z}\cong \Z^{E_P(i)}$. Denoting by
\[
W(iP):=\Set{P\cap(\Z/i)^n \to \R}\cong \R^{E_P(i)},
\]
we identify $W(iP)$ with its dual space by the scalar product
\[
\bracket{\varphi, \psi}=\sum_{\bs a\in A}\varphi(\bs a)\psi(\bs a).
\]
Following GKZ theory \cite[p.$221$, $(1.6)$]{GKZ94}, we define a concave piecewise linear function on $P$ as follows.
We also refer the reader to Section $4.2$ in \cite{Don02} for another viewpoint of this argument. 

Let us fix $\varphi\in W(iP)$. Let $G_{\varphi}\subset \R^n\times \R\cong \R^{n+1}$ be the convex hull of the union of vertical half-lines:
\begin{equation}\label{def:graph}
G_{\varphi}:=\conv\Set{\bigcup_{\bs a\in P\cap(\Z/i)^n}\set{(\bs a,t)|t\leqslant \varphi(\bs a)}}.
\end{equation}
Then, the upper boundary of $G_\varphi$ can be regarded as the graph of a piecewise linear function $g_\varphi:P\to \R$. Specifically, we define $g_\varphi$ by
\[
g_\varphi(\bs x)=\max\set{t| (\bs x, t)\in G_\varphi}
\]
for any $\bs x \in P$. Remark that the function $g_\varphi$ is concave by construction.
For later use, we also define the set of piecewise linear functions by
\begin{equation}\label{eq:PLfuns}
PL(P;i)=\set{g_\varphi| \varphi\in W(iP)}.
\end{equation}
In \cite[Definition $13$]{Sean12}, Paul defined the Futaki invariant, which is said to be the {\emph{Futaki-Paul invariant}}, for the notion of P-stability. He also proved that the Futaki-Paul invariant is less than or equal to zero if and only if
a projective variety $X\subset \P^N$ is numerically semistable \cite[Proposition $2.10$]{Sean12}.
We call this criteria the {\emph{Hilbert-Mumford-Paul (HMP) criterion}}.

In the toric case, one can see that the Futaki-Paul invariant $F_P(\varphi)$ can be explicitly written in terms of the moment polytope $P$.
 Specifically, $F_P(\varphi)$ is the difference between the Chow weight and the Hurwitz weight as follows:
 \begin{align}\label{def:FPinv}
 \begin{split}
F_P(\varphi)&:=
\deg(\Hu_{X_P}(i))\max\set{\bracket{\varphi,\psi}|\psi \in \sec(iP)}\\
&\qquad \quad  -\deg(R_{X_P}(i))\max\set{\bracket{\varphi,\psi}|\psi \in \Hu(iP)}
\end{split}
\end{align}
for any $\varphi\in W(iP)$, where $\sec(iP)$ and $\Hu(iP)$ denote the secondary polytope and the Hurwitz polytope of $iP$ respectively. 
Moreover, the HMP criterion in the toric setting is the following.
\begin{proposition}[Paul]\label{prop:HMPcriterion}
For a positive integer $i$, let $P$ and $E_P(i)$ as above. Setting by $N_i:=E_P(i)$, we consider the diagonal subtorus $H(iP)$ of $(\C^\times)^{N_i}$ defined by
\[
H(iP)=\Set{(t_1,\dots, t_{N_i})\in (\C^\times)^{N_i}|\prod_{j=1}^{N_i}t_j=1}.
\]
Then, the pair $\left(R_{X_P}^{\deg \Hu_{X_P}(i) },\Hu_{X_P}^{\deg R_{X_P}(i)}\right)$ is $H(iP)$-semistable if and only if $F_P(\varphi)$ defined in $\eqref{def:FPinv}$
is less than or equal to zero for any $\varphi\in W(iP)$.
\end{proposition}
Hence, it is important to specify the affine hull of $\Hu(P)$ as well as $\max\set{\bracket{\varphi,\psi}|\psi \in \Hu(P)}$.
We shall investigate the combinatorial description of them in Section \ref{sec:KeyProps}.

\section{The key propositions}\label{sec:KeyProps}

The goal of this section is to show the following statement of the affine hull of the Hurwitz polytope of a smooth projective toric variety.
We shall prove Proposition \ref{prop:AffHurwitz} by applying techniques used in the proof of Lemma \ref{lem:degHu}. 
\begin{proposition}\label{prop:AffHurwitz} 
Let $A=\set{\bs a_0, \ldots , \bs a_N}\subset M_\Z$ be a finite set of integer vectors which satisfies $(*)$. Let $P$ denote the convex hull of $A$ in $M_\R$.
Then the affine hull of the Hurwitz polytope $\Hu(P)$ of $P$ in $W(P)\cong \R^{N+1}$ is equal to
\begin{align}\label{eq:AffHurwitz}
\mathrm{Aff}_\R(\Hu(P))  
&=\Set{(\psi_0, \ldots ,\psi_N)\in \R^{N+1}|
\begin{matrix}
\displaystyle \sum_{i=0}^N\psi_i=n(n+1)!\vol(P)-n!\vol(\p P) \\
\displaystyle \sum_{i=0}^N\psi_i\bs a_i=n(n+1)!\int_P\bs x \,dv-n!\int_{\p P}\bs x\,d\sigma 
\end{matrix}
}.
\end{align}
\end{proposition}

\begin{proof}
In the same manner as in the proof of Lemma \ref{prop:AffDiscrim}, it is enough to show that the equalities in $\eqref{eq:AffHurwitz}$ hold on $\Hu(P)$.

The first equality is an immediate consequence of Lemma \ref{lem:degHu}. 

For the second equality, we again consider the $(2n-1)$-dimensional product polytope $Q=P\times \D_{n-1}$ as in the proof of Lemma \ref{lem:degHu}.
For each $0\leqslant j \leqslant n$, we denote the barycenter of $\p^j P$ by $\bs b_{\p^j P}$:
\[
\bs b_{\p^j P}:=\int_{\p^j P}\bs x\, d\sigma^{(j)}=\sum_{F\in \mathcal{F}^j(P)}\bs b_F.
\]
In particular, for $j=0$ and $1$, we see that
\[
\bs b_{\p^0 P}=\bs b_P=\int_P \bs x\, dv \qquad \text{and} \qquad \bs b_{\p P}=\int_{\p P} \bs x\, d\sigma =\sum_{F\in \mathcal{F}(P)}\bs b_F
\]
as usual. 

Let $\widetilde{\bs x}=(\bs x, \bs 1)$ be the coordinate function of $Q$ with the base coordinate $\bs x=(x_1, \dots , x_n)$
and the fiber coordinate $\bs 1=(\,  \underbrace{1, \dots, 1}_{\text{$(n-1)$-th}}\, )$. By adapting Lemma \ref{prop:AffDiscrim} to the product polytope $Q$, we have
\begin{equation}\label{eq:alternateQ}
\sum_{i=0}^N\psi_i \bs a_i= \sum_{j=0}^{2n-1}(-1)^{2n-j}(2n-j)!\int_{\p^j Q}\widetilde{\bs x}\, d\sigma^{(j)}_Q
\end{equation}
by $\eqref{eq:AffDiscrim2}$. Especially, $\eqref{eq:AffDiscrim1}$ gives
\[
\sum_{i=0}^N \psi_i=\sum_{F \preceq Q} (-1)^{\codim F}(\dim F+1)!\vol(F)=n(n+1)!\vol(P)-n!\vol(\p P),
\]
which was already proved in Lemma \ref{lem:degHu}.
Similarly, a straightforward computation shows that
\begin{align*}
\int_Q \widetilde{\bs x}\, dv_Q&=\int_P \bs x\, dv_P\cdot \vol(\D_{n-1})={{n}\choose{n}}\cdot \frac{1}{(n-1)!}\bs b_P, \\
\int_{\p Q} \widetilde{\bs x}\, d\sigma_Q&={{n}\choose{n-1}}\cdot \frac{1}{(n-2)!}\bs b_P+{{n}\choose{n}}\cdot \frac{1}{(n-1)!}\bs b_{\p P}, \\
\int_{\p^2 Q} \widetilde{\bs x}\, d\sigma^{(2)}_Q&={{n}\choose{n-2}}\cdot \frac{1}{(n-3)!}\bs b_P+{{n}\choose{n-1}}\cdot \frac{1}{(n-2)!}\bs b_{\p P}+{{n}\choose{n}}\cdot \frac{1}{(n-1)!}\bs b_{\p^2 P}, \\
\vdots & \\
\sum_{\widetilde{\bs a}\in \mathcal V(Q)}\widetilde{\bs a}&=\#\mathcal V(\D_{n-1})\sum_{\bs a \in \mathcal V(P)}\bs a=n\sum_{\bs a \in \mathcal V(P)} \bs a.
\end{align*}
Plugging these values into $\eqref{eq:alternateQ}$, we find that
\begin{align*}
\mathrm{Coeff}(\bs b_P)&=(n+1)!\sum_{k=1}^n(-1)^{n-k}{{n}\choose{k}}{{n+k}\choose{k-1}}=n(n+1)!  ,\\
\mathrm{Coeff}(\bs b_{\p P})&=-n!\sum_{k=1}^n(-1)^{n-k}{{n}\choose{k}}{{n+k-1}\choose{k-1}}=-n!  , ~~\text{and} \\
\mathrm{Coeff}(\bs b_{\p^j P})&=0, \quad \text{for} \quad 2\leqslant j \leqslant n,
\end{align*}
with the aid of Claim \ref{claim:binomial}. This completes the proof of Proposition \ref{prop:AffHurwitz}. 
\end{proof}
The following example illustrates how the affine hull of the discriminant 
polytope can be determined by alternative summands of barycenters for each codimension $j$ faces of $P$.
\begin{example}\label{ex:A-discrim}\rm
Let $A$ be the set of lattice points in $\Z^2$ given by
\[
\Set{\begin{pmatrix}
0 \\ 0
\end{pmatrix},
\begin{pmatrix}
0 \\ 1
\end{pmatrix},
\begin{pmatrix}
0 \\ 2
\end{pmatrix},
\begin{pmatrix}
1 \\ 0
\end{pmatrix},
\begin{pmatrix}
1 \\ 1
\end{pmatrix},
\begin{pmatrix}
2 \\ 0
\end{pmatrix}
}=\Set{\bs{a}_0, \dots ,\bs{a}_5}.
\]
Let $P$ be the convex hull of $A$ in $\R^2$ which is the standard triangle with vertices $\bs a_0$, $\bs a_2$ and $\bs a_5$. 
It is known that the discriminant $\D_{X_P}$ is then given by (see, \cite[p. $437$, $(1.13)$]{GKZ94})
\begin{equation}\label{eq:QuadForm}
\D_{X_P}=a_{00}a_{11}^2+a_{01}^2a_{20}+a_{02}a_{10}^2-a_{01}a_{10}a_{11}-4a_{00}a_{02}a_{20}.
\end{equation}
Hence, we see that the discriminant polytope in $\R^6$ is
\begin{equation}\label{eq:QuadDiscrim}
\Wt_{(\C^\times)^6}(\D_{X_P})=\conv\Set{
\begin{pmatrix}
1 \\ 0 \\ 0 \\ 0 \\ 2 \\ 0
\end{pmatrix},
\begin{pmatrix}
0 \\ 2 \\ 0 \\ 0 \\ 0 \\ 1
\end{pmatrix},
\begin{pmatrix}
0 \\ 0 \\ 1 \\ 2 \\ 0 \\ 0
\end{pmatrix},
\begin{pmatrix}
0 \\ 1 \\ 0 \\ 1 \\ 1 \\ 0
\end{pmatrix},
\begin{pmatrix}
1 \\ 0 \\ 1 \\ 0 \\ 0 \\ 1
\end{pmatrix}
}.
\end{equation}
We remark that these five vertices of $\Wt_{(\C^\times)^6}(\D_{X_P})$ are corresponding to five different $D$-equivalence classes of $14$ regular triangulations of $(P,A)$.
See, \cite[Example $2.3$]{KL20} for the associated massive GKZ vectors as well as their ordinary GKZ vectors. Then, $\eqref{eq:degDiscrim}$ gives that
\[
\deg(\D_{X_P})=(2+1)!\cdot \vol(P)-2!\cdot 3\cdot 2+3=3,
\]
which agrees with the result of $\eqref{eq:QuadForm}$. Moreover, we see that
\[
\bs b_P=\int_P \bs x\,dv=\left(\frac{4}{3}, \frac{4}{3} \right), \quad \bs b_{\p P}=\int_{\p P} \bs x\,d\sigma=\left( 4, 4 \right), \quad \text{and}\quad 
\sum_{\bs a \in \mathcal V(P)} \bs a=(2,2)
\]
by direct computations. Thus, the right-hand side of $\eqref{eq:AffDiscrim2}$ becomes
\[
3!\left(\frac{4}{3}, \frac{4}{3} \right)-2!\left(4,4 \right)+(2,2)=(2,2),
\]
whereas the left-hand side of $\eqref{eq:AffDiscrim2}$ is
\[
\sum_{i=0}^5\eta_i \bs a_i=
\begin{pmatrix}
\eta_3+\eta_4+2\eta_5 \\
\eta_1+2\eta_2+\eta_4
\end{pmatrix}.
\]
This shows that the affine hull of $\eqref{eq:QuadDiscrim}$ is determined by
\[
\sum_{i=0}^5\eta_i \bs a_i=\sum_{j=0}^2\sum_{F\in \mathcal{F}^j(P)}(-1)^{\codim F}(\dim F+1)!\int_F\bs x\, d\sigma_F.
\]
\end{example}
We finish this section with the following proposition.
\begin{proposition}\label{prop:MaxHurwitz}
We have the following equality on $\Hu(P)$:
\[
\max\Set{\bracket{\vp, \psi}|\psi \in \Hu(P)}=n(n+1)!\int_Pg_\vp(\bs x)dv-n!\int_{\p P}g_\vp(\bs x)d\sigma.
\]
\end{proposition}
\begin{proof}
For any triangulation $T$ of $(P, A)$, let $\xi_T=n\eta_{T,n}-\eta_{T,n-1}$ be the Hurwitz vector defined in $\eqref{def:HurVec}$. Then, we have
\begin{align*}
\max\set{\braket{\vp, \psi}| \psi\in \Hu(P)}&=\max_{T\preceq  (P,A)}\set{\braket{\vp, \xi_T}| \xi_T\in \Hu(P)}\\
&=\max_{T\preceq  (P,A)}\Set{n(n+1)!\int_Pg_{\vp, T}(\bs x)dv-n!\int_{\p P}g_{\vp, T}(\bs x)d\sigma}
\end{align*}
by Lemma \ref{lem:alternate}. Moreover, for any $\psi\in \Hu(P)$, we observe that
\begin{equation}\label{ineq:Hurwitz}
\braket{\vp,\psi}\leqslant n(n+1)!\int_P g_\vp(\bs x)dv-n!\int_{\p P}g_\vp(\bs x)d\sigma
\end{equation}
from Proposition \ref{prop:AffHurwitz} and the argument given in the proof of Lemma \ref{lem:alternate}.

In order to show the equality in $\eqref{ineq:Hurwitz}$, it suffices to provide a triangulation $T$ such that $g_\vp=g_{\vp,T}$.
For this, we firstly consider the projection of the roof of $G_\vp$ into $P$, where $G_\vp$ is convex polyhedron defined in $\eqref{def:graph}$.
Then the projection $p_\vp:G_\vp \to P$ gives a subdivision $\mathcal S=\Set{S_1, \dots , S_m}$ of $(P,A)$.
Note that for each $1\leqslant j \leqslant m$, $P_j=\conv(S_j)$ determines a polytope with vertices in $A$ but $P_j$ may not be a simplex.

Now we take a generic $\vp'$ close to $\vp$. Let $p_{\vp'}:G_{\vp'}\to P$ be the projection from the roof of $G_{\vp'}$ into $P$.
Since $\vp'$ is generic, the projection $p_{\vp'}$ gives a triangulation $T$ of $(P,A)$.
Moreover, $p_{\vp'}$ induces a triangulation of each polytope $P_j$ because $\vp'$ is close enough to $\vp$.
Consequently, we see that $g_\vp$ is a $T$-piecewise linear function, and hence $g_\vp=g_{\vp,T}$.
This completes the proof of Proposition \ref{prop:MaxHurwitz}. 
\end{proof}

\section{Proof of the main theorem}\label{sec:Proof}
In this section, we prove Theorem $\ref{thm:main}$.

\begin{proof}[Proof of Theorem $\ref{thm:main}$]
By Proposition $\ref{prop:HMPcriterion}$, we have
\begin{equation}\label{ineq:Pss}
\deg(\D_{X_P}(i))\max \Set{\bracket{\vp, \psi}| \psi \in \mathrm{Ch}(iP)}\leqslant \deg(R_{X_P}(i))\max \Set{\bracket{\vp, \psi}| \psi \in \Hu(iP)}
\end{equation}
for any $\vp\in W(iP)$. Then, Lemma $1.9$ in \cite[p. $221$]{GKZ94} gives the equality 
\[
\max \Set{\bracket{\vp, \psi}| \psi \in \mathrm{Ch}(iP)}=i^n(n+1)!\int_P g_\vp(\bs x)dv
\]
for any $\vp\in W(iP)$. Hence, the left-hand side of $\eqref{ineq:Pss}$ is equal to
\begin{align*}
\Set{i^nn(n+1)!\, \vol(P)-i^{n-1}n!\, \vol\bigl(\p P \bigr)}\cdot\left(i^n(n+1)!\int_{P}g_\vp(\bs x)dv \right).
\end{align*}
On the other hand, Proposition \ref{prop:MaxHurwitz} yields that the right-hand side of $\eqref{ineq:Pss}$ is 
\[
(n+1)!\, i^n\vol(P)\Set{i^nn(n+1)!\int_Pg_\vp(\bs x)dv-i^{n-1}n!\int_{\p P}g_\vp(\bs x)d\sigma}.
\]
By simplifying the coefficient of $\int_P g_\vp(\bs x)dv$ in $\eqref{ineq:Pss}$, we see that
\[
\mathrm{Coeff}\left(\int_P g_\vp(\bs x)dv \right)=i^{2n-1}(n+1)!\cdot n!\, \vol(\p P).
\]
Similarly, we have  
\[
\mathrm{Coeff}\left(\int_{\p P} g_\vp(\bs x)d\sigma \right)=-i^{2n-1}(n+1)!\cdot n!\, \vol(P).
\]
Consequently we find that
\begin{align*}
\eqref{ineq:Pss} \quad \Leftrightarrow \quad i^{2n-1}(n+1)!\, n!\left(\vol(\p P)\int_Pg_\vp(\bs x)dv-\vol( P)\int_{\p P}g_\vp(\bs x)d\sigma \right)\geqslant 0
\end{align*}
for any $\vp\in W(iP)$. Combining this and Kempf's instability theorem (Proposition \ref{prop:Kempf}), we complete
the proof of Theorem \ref{thm:main}. See also, $\eqref{eq:PLinv}$ and Corollary \ref{cor:main}.
\end{proof}
For an $n$-dimensional smooth projective toric variety $X$, Let $\mathrm{Aut}(X)$ be the automorphism group of $X$ and $T_\C\cong (\C^\times)^n$ be the algebraic torus which defines the toric structure of $X$.
Then, the Demazure's structure theorem \cite[p. 140]{Od88} states that $T_\C$ is the maximal algebraic torus in $\mathrm{Aut}(X)$. 
From the definition of a toric variety, the algebraic torus $T_\C$ acting on $X$ has an open dense orbit $U\subset X$.
Then, the normalizer $\mathcal N(T_\C)\subset \mathrm{Aut}(X)$ of $T_\C$ has a natural action on $U$.
Setting $\mathcal W(X)=\mathcal N(T_\C)/T_\C$, we call $\mathcal W(X)$ the {\emph{Weyl group}}.

By choosing an arbitrary point $\bs x_0\in U$, we identify $U$ with $T_\C$.
This identification gives a splitting of the short exact sequence
\begin{equation*}
\xymatrix@R=-1ex{
1~~\ar[r]&~~T_\C ~~\ar[r]&~~\mathcal N(T_\C)~~\ar[r]&~~\mathcal W(X)~~\ar[r]&~~1.
 }
\end{equation*}
Thus, we obtain an embedding $\mathcal W(X) \hookrightarrow \mathcal N(T_\C) \subset \mathrm{Aut}(X)$.
Let us denote by $\mathcal K(T_\C)\cong (S^1)^n$ the maximal compact subgroup in $T_\C$.
Then, there is the canonical isomorphism $U/\mathcal K(T_\C)\cong N_\R$, where $N_\R$ is the dual space of $M_\R$.
In particular, the associated fan $\Sigma$ of $X$ is in $N_\R$. Then, the action of $\mathcal N(T_\C)$ on $U$ induces the linear action of $\mathcal W(X)$ on $N_\R$.

Let $(X_P, L_P)$ be a polarized smooth toric variety with the associated moment polytope $P$. 
For any positive integer $i$, let $\PL(P;i)$ be the set of piecewise linear functions defined in $\eqref{eq:PLfuns}$.
There is the natural action of $\mathcal W(X_P)$ on $\PL(P;i)$ and we denote $\PL(P;i)^{\mathcal W(X_P)}$
the set of the Weyl group invariant piecewise linear functions in $\PL(P;i)$:
\begin{equation}\label{eq:PLinv}
\PL(P;i)^{\mathcal W(X_P)}=\Set{g\in \PL(P;i)|g(w\cdot \bs x)=g(\bs x)  \text{ for any }   w\in \mathcal W(X_P) }.
\end{equation}
Combining with Theorem $\ref{thm:main}$ and 
Proposition $\ref{prop:Kempf}$, we obtain the following corollary.
\begin{corollary}\label{cor:main}
Let $P$ be a Delzant lattice polytope in $M_\R$. Then for any integer $i>0$, $(X_P,L_P^{\otimes i})$ is numerically semistable if and only if $\eqref{ineq:FPinv}$ holds for any $g\in \PL(P;i)^{\mathcal W(X_P)}$.
\end{corollary}
\begin{remark}\rm
It is remarkable that we only need to consider a concave $T$-piecewise linear function $g_\vp$ corresponding to a {\emph{regular triangulation}} of $P$ in Corollary $\ref{cor:main}$.
In other words, in order to verify numerical semistability of a polarized toric variety, it is enough to consider all $T$-piecewise linear functions which correspond to the vertices of the secondary polytope $\sec(P)$.
(see, p. $221$, Theorem $1.7$ (b) in \cite{GKZ94}).
\end{remark}
In particular, if we take $g_\vp$ to be linear functions in $\eqref{ineq:FPinv}$, we have the following necessary condition for numerical semistability.
\begin{corollary}
Let $P$ be a Delzant lattice polytope in $M_\R$, and let $(X_P, L_P)$ be the associated smooth polarized toric variety. We assume that $(X_P, L_P^{\otimes i})$ is numerically semistable for some integer $i>0$.
Then the following equality holds:
\[
\vol(\p P)\int_P \bs x\, dv-\vol( P)\int_{\p P} \bs x \,d\sigma =\bs 0.
\]
\end{corollary}
\begin{example}[Ono's theorem of Chow semistability for toric varieties]\label{ex:Ono}
\rm
If we consider a special case of numerical semistability, we reproof Ono's theorem of Chow semistability for polarized toric varieties as follows.

For a lattice Delazant polytope $P$, let $R_{X_P}(i)$ denote the Chow form of $(X_P, L_P^{\otimes i})$.
Let us consider the pair of trivial representations $\left( \bs 1^{\deg R_{X_P}(i)}, R_{X_P}(i)^{N_i}\right)$, where $N_i$ is the number of lattice points in $iP$.
Then, we define the function $d_{iP}\in W(iP)$ by $d_{iP}(\bs a)=1$ for all $\bs a\in P\cap(\Z/i)^n$. Let $H(iP)$ be the standard torus of $(\C^\times)^{N_i}$ defined by
\[
H(iP)=\Set{(t_1, \dots ,t_{N_i})\in (\C^\times)^{N_i} | \prod_{j=1}^{N_i}t_j=1}.
\]
Applying Thereom $\ref{thm:main}$ into $\left( \bs 1^{\deg R_{X_P}(i)}, R_{X_P}(i)^{N_i}\right)$, we see that the Chow form of $(X_P, L_P^{\otimes i})$ is $H(iP)$-semistable if and only if
\begin{align}
&\left( \bs 1^{\deg R_{X_P}(i)}, R_{X_P}(i)^{N_i}\right) \text{ is numerically semistable with respect to $H(iP)$-action}  \notag \\
\Leftrightarrow \quad & \deg \left(R_{X_P}(i) \right) \max\Set{\braket{\vp,\psi}| \psi=d_{iP}}\leqslant N_i\max\Set{\braket{\vp,\psi}| \psi\in \Ch(iP)} \text{ for any } \vp\in W(iP) \notag \\
\Leftrightarrow \quad & \frac{(n+1)!\vol(iP)}{N_i}\bracket{\vp, d_{iP}}\leqslant \max\Set{\braket{\vp,\psi}| \psi\in \Ch(iP)} \text{ for any } \vp\in W(iP). \label{ineq:Chowss}
\end{align}
By Lemma $1.9$ in \cite[Chapter $7$]{GKZ94} and Lemma $3.3$ in \cite{ZZ08}, we find that
\[
 \max\Set{\braket{\vp,\psi}| \psi\in \Ch(iP)} =(n+1)!\int_{iP}g_\vp(\bs x)dv=i^n(n+1)!\int_P g_\vp(\bs x)dv.
\]
Since
\[
\bracket{\vp, d_{iP}}=\sum_{\bs a\in P\cap (\Z/i)^n}\vp(\bs a)d_{iP}(\bs a)=\sum_{\bs a\in P\cap (\Z/i)^n}g_{\vp}(\bs a),
\]
we conclude that
\begin{align*}
\eqref{ineq:Chowss} \quad & \Leftrightarrow \quad \int_P g_\vp(\bs x)dv \geqslant \frac{\vol (P)}{N_i} \sum_{\bs a\in P\cap (\Z/i)^n}g_{\vp}(\bs a) 
\end{align*}
for any $g_\vp \in \PL(P;i)^{\mathcal W(X_P)}$, which is a necessary and sufficient condition for Chow semistability of polarized toric varieties firstly proved by Ono \cite{Ono13} (see, also Theorem $2.6$ in \cite{LLSW19}).
\end{example}

\section{The relationship among other notions of GIT-stability}\label{sec:Relationship}
When a reductive algebraic group $G$ acts on an algebraic variety $X$ linearly, we have several steps for constructing the quotient according to the classical GIT.
At first we choose a $G$-equivariant embedding of $X$ into a certain projective space with the lifted action $\widetilde G$ induced by $G\curvearrowright X$.
This corresponds to the choice of a linearization of the action, and it is defined by a $G$-linearized ample line bundle on $X$.
Due to different choices of linearization of an ample line bundle, there are many notions of GIT-stability for algebraic varieties. Among various notions of GIT-stability,
K-stability and Chow stability are most widely studied in K\"ahler geometry. In this section, we clarify the relationship among P-stability, Chow stability and K-stability for a polarized toric variety.

\subsection{Asymptotic numerical semistability and K-semistability}\label{sec:PandK}
Firstly, we prove that asymptotic numerical semistability is equivalent to K-semistability in the sense of Donaldson \cite{Don02} for a polarized toric manifold.
\begin{theorem}\label{thm:PandK}
Let $(X_P, L_P)$ be a smooth polarized toric variety with the associated moment polytope $P$.
Then $(X_P,L_P)$ is asymptotically numerically semistable if and only if it is K-semistable for toric degenerations.
\end{theorem}
\begin{proof}
Obviously, K-semistability implies asymptotic numerical semistability because
\[
\vol(\p P)\int_Pg_{\varphi}(\bs x)dv-\vol(P)\int_{\p P}g_{\varphi}(\bs x)d\sigma\geqslant 0
\]
holds for {\em any} concave piecewise linear function.

We shall prove the converse implication by contradiction.
Since $(X_P,L_P)$ is asymptotically numerically semistable, we see that
\begin{equation}\label{ineq:AsymPss}
F_P(i;g_{\varphi,T})=\vol(\p P)\int_Pg_{\varphi,T}(\bs x)dv-\vol(P)\int_{\p P}g_{\varphi,T}(\bs x)d\sigma \geqslant 0
\end{equation}
holds for all $i\gg0$ and and each $g_{\varphi,T}\in \PL(P;i)^{\mathcal W(X_P)}$, where $T$ is a regular triangulation of $(P,A)$.
Now we suppose that $(X_P, L_P)$ is K-unstable.
Hence, there is a concave piecewise linear function $g_{\varphi,T'}\in \PL(P;i)^{\mathcal W(X_P)}$ with a non-regular triangulation $T'$ such that $F_P(i;g_{\varphi,T'})<0$.
Recall that $g_{\varphi,T}$ is a unique $T$-piecewise linear function determined by $g_{\varphi,T}(\bs a_i)=\varphi_i$ as we saw in Section \ref{sec:SecPoly}.
In particular, we have $g_{\varphi,T}(\bs x)\leqslant g_{\varphi,T'}(\bs x)$ for any point $\bs x$ in $P$ (see, \cite[$(2.3)$]{BFS90}).
By integrating both sides over $P$, we see that
\begin{equation}\label{ineq:BFS}
\int_Pg_{\varphi,T}(\bs x)dv \leqslant \int_P g_{\varphi,T'}(\bs x)dv.
\end{equation}
Since $F_P(i;g_{\varphi,T'})<0$, we see that
\begin{equation}\label{ineq:K_unsta}
 \int_P g_{\varphi,T'}(\bs x)dv<\frac{\vol(P)}{\vol(\p P)} \int_{\p P}g_{\varphi,T'}(\bs x)d\sigma.
\end{equation}
On the other hand, our assumption $\eqref{ineq:AsymPss}$ yields that
\begin{equation}\label{ineq:AsymPss2}
 \int_P g_{\varphi,T}(\bs x)dv\geqslant \frac{\vol(P)}{\vol(\p P)} \int_{\p P}g_{\varphi,T}(\bs x)d\sigma.
\end{equation}
Comparing the left hand sides of $\eqref{ineq:K_unsta}$ and $\eqref{ineq:AsymPss2}$ using $\eqref{ineq:BFS}$,
we conclude that
\[
\int_{\p P}g_{\varphi,T}(\bs x)d\sigma < \int_{\p P}g_{\varphi,T'}(\bs x)d\sigma.
\]
This implies the strict inequality 
\begin{equation}\label{ineq:assump}
g_{\varphi,T}(\bs x)<g_{\varphi,T'}(\bs x)
\end{equation}
for any $\bs x\in P$. There are two possibilities for the relation between the triangulations $T$ and $T'$:
(i) $T$ is a refinement of $T'$, and (ii) $T$ is not a refinement of $T'$. 

\vskip 4pt

(i) If $T$ is a refinement of $T'$, we prove the assertion as follows: Firstly, we see that
\begin{equation}\label{ineq:assump2}
\int_{P}g_{\varphi,T}(\bs x)dv < \int_{P}g_{\varphi,T'}(\bs x)dv
\end{equation}
by integrating both sides of $\eqref{ineq:assump}$ over $P$.
Since $g_{\varphi,T'}$ is linear on each simplex $C'$ in $T'$, we see the equality
\begin{equation}\label{eq:IntDom}
\int_{C'}g_{\varphi,T'}(\bs x)dv=\vol(C')g_{\varphi,T'}(\bs b_{C'})
\end{equation}
as in $\eqref{eq:lin_bary2}$, where $\bs b_{C'}$ is the barycenter of a simplex $C'$.
Then, $\eqref{eq:IntDom}$ yields that
\begin{align}
\begin{split}\label{eq:Decomp}
\int_{P}g_{\varphi,T'}(\bs x)dv&=\sum_{C'\in T'}\int_{C'}g_{\varphi,T'}(\bs x)dv\\
&=\sum_{C'\in T'}\vol(C')g_{\varphi,T'}(\bs b_{C'}).
\end{split}
\end{align}
Since each $C'$ is an $n$-dimensional simplex, the barycenter $\bs b_{C'}$ is given by the average of its vertices:
\[
\bs b_{C'}=\frac{1}{n+1}\sum_{\bs a\in \mathcal V(C')}\bs a.
\]
Plugging this into $\eqref{eq:Decomp}$, we find that
\begin{equation}\label{eq:Decomp3}
\int_Pg_{\varphi,T'}(\bs x)dv=\sum_{C'\in T'}\frac{\vol(C')}{n+1}\sum_{j\in C'}\varphi_j.
\end{equation}
Since the same formula holds for the triangulation $T$, we have the equality
\begin{equation}\label{eq:Decomp2}
\int_Pg_{\varphi,T}(\bs x)dv=\sum_{C\in T}\frac{\vol(C)}{n+1}\sum_{i\in C}\varphi_i.
\end{equation}
Note that each simplex $C'$ of $T'$ is written as
\[
C'=C_{i_1}\cup\dots \cup C_{i_\ell}
\]
using appropriate simplices $C_{i_k}$ of $T$ with the index subset $\set{i_1, \dots, i_\ell}\subset \set{0, \dots, N}$, because $T$ is a refinement of $T'$. 
This yields that the right hand side of $\eqref{eq:Decomp2}$ equals
\[
\sum_{C'\in T'}\frac{\vol(C')}{n+1}\sum_{j\in C'}\varphi_j.
\]
Combining this and $\eqref{eq:Decomp3}$, we derive a contradiction from $\eqref{ineq:assump2}$

\vskip 4pt

(ii) If $T$ is not a refinement of $T'$, there exists a point $\bs a\in A$ such that $\bs a$ is a vertex of simplex $C'$ from $T'$, but not a vertex of any simplex $C$ from $T$.
Considering a unique $T'$-piecewise linear function $h_{\varphi, T'}$ determined by $h_{\varphi, T'}(\bs a_i)=\varphi_i$ once again,
one can see that $h_{\varphi, T}(\bs a)< h_{\varphi, T'}(\bs a)$ for each $\bs a\in A$. By Definition $1.4$ (b) in \cite[p. $219$]{GKZ94}, we have $h_{\varphi, T}\geqslant \varphi(\bs a)$
because $\bs a$ is not a vertex of any simplex from $T$. Then, this implies that
\[
\varphi(\bs a)\leqslant h_{\varphi, T}(\bs a)<h_{\varphi, T'}(\bs a)=\varphi(\bs a),
\]
which yields the contradiction.
\end{proof}

\subsection{Chow stability}
Next we consider Chow stability for a polarized toric variety. 
In \cite{Yotsu16}, it was proved by the author that the Chow form of a (not-necessarily-smooth) projective toric variety is $H$-semistable if and only if it is $H$-polystable with respect to the standard 
complex torus action $H$. Combining with this result and Proposition \ref{prop:Kempf}, we have the following.
\begin{proposition}
Let $A=\set{\bs a_0, \dots ,\bs a_N}\subset M_\Z$ be a finite set of lattice points satisfying $(*)$, and let $P$ be its convex hull.
Then, the associated projective toric variety $X_P\subset \P^N$ is Chow semistable if and only if it is Chow polystable.
\end{proposition}
Finally, we discuss the relation between Chow semistability and numerical semistability. Let us fix a positive integer $i$.
It is natural to expect that Chow semistability implies numerical semistability for an $n$-dimensional smooth projective toric variety $(X_P, L_P^{\otimes i})$.
However, there are some technical issues that prevent us from proving the implication as explained below:
suppose that a polarized toric variety $(X_P,L_P^{\otimes i})$ is Chow semistable for a fixed integer $i>0$.
By Ono's theorem dealt in Example \ref{ex:Ono}, we have
\[
\mathrm{Chow}_P(i;g_\vp):=E_P(i)\int_Pg_{\varphi}(\bs x)\,dv-\vol (P) \!\!\sum_{\bs a\in P\cap (\Z/i)^n}g_{\vp}(\bs a) \geqslant 0
\]
for any $g_\vp \in \PL(P;i)^{\mathcal W(X_P)}$. By the Euler-Maclaurin summation formula \cite{BV97} (see also, Lemma 3.3 in \cite{ZZ08}),
we find that
\[
\mathrm{Chow}_P(i;g_\vp)=\frac{i^{n-1}}{2}F_P(i;g_\vp)+ O(i^{n-2}),
\]
where $F_P(i;g_\vp)$ is the Futaki-Paul invariant in $\eqref{ineq:FPinv}$.
Hence, if $i\in \Z_{>0}$ is sufficiently large, then we conclude that $F_P(i;g_\vp)\geqslant 0$, that is, $(X_P,L_P^{\otimes i})$ is numerically semistable. 
However, for the {\em{fixed}} integer $i>0$, the sign of the leading term $F_P(i;g_\vp)$ is inconclusive even though $\mathrm{Chow}_P(i;g_\vp)\geqslant 0$.
Thus, we fail to prove that Chow semistability implies numerical semistability with respect to the fixed polarization $(X_P, L_P^{\otimes i})$ at the moment.


\end{document}